\documentclass[a4paper,11pt]{article}
\usepackage[margin=2.5cm]{geometry}
\usepackage{graphicx} 
\usepackage{amsmath,amssymb,amsfonts,amsthm,mathtools,thmtools}
\numberwithin{equation}{section}
\usepackage{todonotes, hyperref}
\usepackage{cleveref}
\usepackage{bm, dsfont} 

\usepackage{tikz}
\usepackage{pgfplots}
\pgfplotsset{compat=1.18} 

\renewcommand{\d}{\, \mathrm{d}}
\renewcommand{\Im}{\mathrm{Im}}
\renewcommand{\Re}{\mathrm{Re}}
\newcommand{\e}{\mathrm{e}}
\newcommand{\im}{\mathrm{i}}
\newcommand{\R}{\mathbb{R}}
\newcommand{\N}{\mathbb{N}}
\newcommand{\Z}{\mathbb{Z}}
\newcommand{\F}{\mathcal{F}}
\newcommand{\C}{\mathbb{C}}

\DeclareMathOperator{\sinc}{sinc}
\DeclareMathOperator{\sign}{sign}
\newcommand{\norm}[1]{\left\lVert \smash{#1} \right\rVert}

\newtheorem{theorem}{Theorem}[section]
\newtheorem{lemma}[theorem]{Lemma}
\newtheorem{corollary}[theorem]{Corollary}

\theoremstyle{definition}

\newtheorem{remark}[theorem]{Remark}

\title{Approximating the Fourier Transform from Non-equispaced Discrete Samples}
\author{ 
  \href{https://orcid.org/0000-0003-3651-4364}{\includegraphics[scale=0.06]{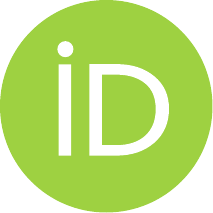}\hspace{1mm}Daniel Potts} \\
    Faculty of Mathematics\\
    Chemnitz University of Technology\\
    Germany\\
    \texttt{potts@math.tu-chemnitz.de}
    \and
	\href{https://orcid.org/0000-0001-7988-1485}{\includegraphics[scale=0.06]{orcid.pdf}\hspace{1mm}Laura Weidensager\thanks{Corresponding author}} \\
	Department of Mathematics\\
	Simon Fraser University\\
	Canada\\
	\texttt{laura\_weidensager@sfu.ca}
}
\date{\today}

\begin{document}

\maketitle

 \begin{abstract}We study the approximation of the Fourier transform of a function from finitely many samples. Departing from the equispaced setting, we sample the function at deterministic non-equispaced nodes obtained by transforming equispaced points on $[0,1]$
through the inverse cumulative distribution function of a probability density. This change of variables compactifies the real line, so that no truncation
of the space domain is necessary. Combining an exact aliasing identity for the midpoint rule with
stationary-phase estimates for the transformed oscillatory integrals, we derive
deterministic error bounds for every frequency and in $L_p$. \par
For functions with polynomial decay in space and frequency, an explicitly optimized density recovers the equispaced convergence
rate, while an additional variance parameter
substantially reduces the pre-asymptotic error constants. For (sub-)exponentially
decaying functions a polynomially decaying density yields (sub-)exponential
rates. Numerical experiments confirm the theory and demonstrate the reduced
pre-asymptotic error compared with optimally scaled equispaced sampling.
\end{abstract}
\vspace{1em} 

\section{Introduction}

Recent advances in deriving error estimates between the discrete Fourier transform and the Fourier transform of functions on $\R$ have highlighted an optimal recipe of how to relate the number of samples, the sampling interval and the grid size,~\cite{EhGrKl24}.
Motivated by the success of random Fourier features~\cite{RaRe07,LiXiYuSu22,PoWei25}, we extend the philosophy from~\cite{EhGrKl24} to non-equispaced samples.
In this work, we investigate whether non-equispaced sampling strategies can 
outperform classical equispaced sampling for functions with specific decay 
properties. This question arises from the observation that the efficiency of 
equispaced sampling may depend on the decay characteristics of the function. 
To address this, we introduce a non-equispaced sampling strategy for the 
approximation of the Fourier transform. By utilizing a tailored transformation, 
we construct a sampling distribution that is explicitly driven by the algebraic 
or exponential decay of the function $f$ and its Fourier transform~$\hat f$.
\par
We study the numerical approximation of the Fourier transform
\begin{equation*}
    \hat f(\xi) = \int_{-\infty}^\infty f(x)\e^{-2\pi\im x\xi}\d x
\end{equation*}
from sample values of $f$.
In~\cite{EhGrKl24} the authors show how the fast Fourier transform (FFT) approximates this integral from equispaced samples of $f$ on a grid,
\begin{equation}\label{eq:f_hat_approx_equi}
    \hat f\left(\frac kp\right) \approx h\sum_{-\tfrac n2<j\leq \tfrac n2}f(hj)\e^{-2\pi\im \frac{kj}{n}},\quad k\in \Z, -\tfrac n2<k\leq \tfrac n2,
\end{equation}
where $p=hn$. This approximation is motivated by discretizing the integral and then truncating the infinite series. The FFT is one of the most important numerical algorithms. One can provide reasonable error estimates for this approximation. For several classes of functions the authors in~\cite{EhGrKl24} answer the question of how the parameters $h,n$ and $p$ should be chosen and give error rates and the asymptotic decay. In contrast, in this paper we investigate the approximation from non-equispaced samples, where more samples are used at points where the function is expected to be larger and less samples in the tails, where the function is decaying. Instead of the approximation~\eqref{eq:f_hat_approx_equi}, we study the approximation of the form
\begin{equation}\label{eq:new_approach}
    \hat f(\xi) \approx \frac 1M \sum_{j=1}^M \frac{f(x_j)}{\mu(x_j)}\e^{-2\pi\im x_j\xi}, \quad x_j\sim \mu,
\end{equation}
where the samples $x_j$ are distributed on $\R$ according to some distribution $\mu$, instead of equispaced $x_j = hj$ as in~\eqref{eq:f_hat_approx_equi}.
For random nodes $x_j\sim\mu$, the sum in~\eqref{eq:new_approach} is the classical Monte Carlo importance sampling estimator of the integral, which is at the heart of random Fourier feature methods~\cite{RaRe07,LiXiYuSu22}. In contrast, we construct the nodes $x_j$ deterministically by inverse transform sampling, i.e., by mapping equispaced points on $[0,1]$ through the inverse cumulative distribution function of $\mu$. This allows us to prove deterministic error bounds instead of estimates that hold only in expectation or with high probability.
\par

The numerical integration of highly oscillatory or heavy-tailed functions over the unbounded domain $\R$ remains a fundamental challenge in computational mathematics. While classical optimally scaled equispaced grids and sinc-based quadratures~\cite{EhGrKl24} achieve high convergence rates for rapidly decaying integrands, their efficacy degrades severely when applied to functions with algebraic or rational tails.
Classical equispaced grids necessitate a truncation domain. This domain stretching imposes a computational compromise. It inevitably dilutes the spatial resolution near the origin, triggering severe phase-aliasing for high-frequency components and increasing the pre-asymptotic error constant. Consequently, while these equispaced schemes may theoretically attain optimal asymptotic rates as the number of nodes $M\rightarrow \infty$, they may fail to provide acceptable accuracy in practically feasible computational regimes.\par
To reduce these pre-asymptotic error bounds, this paper investigates a transformed quadrature scheme. By utilizing a  density function $\mu$ to define a transformation, we automatically compactify the infinite domain $\R$ onto a finite integration window $[0,1]$. This completely eliminates the need for explicit space domain truncation. The transformation naturally redistributes the computational grid. It maintains high resolution at the center to resolve high-frequency oscillations, while it also integrates the asymptotic tails to infinity.
\begin{figure}[htb]
    \centering
    \includegraphics[width=1\linewidth]{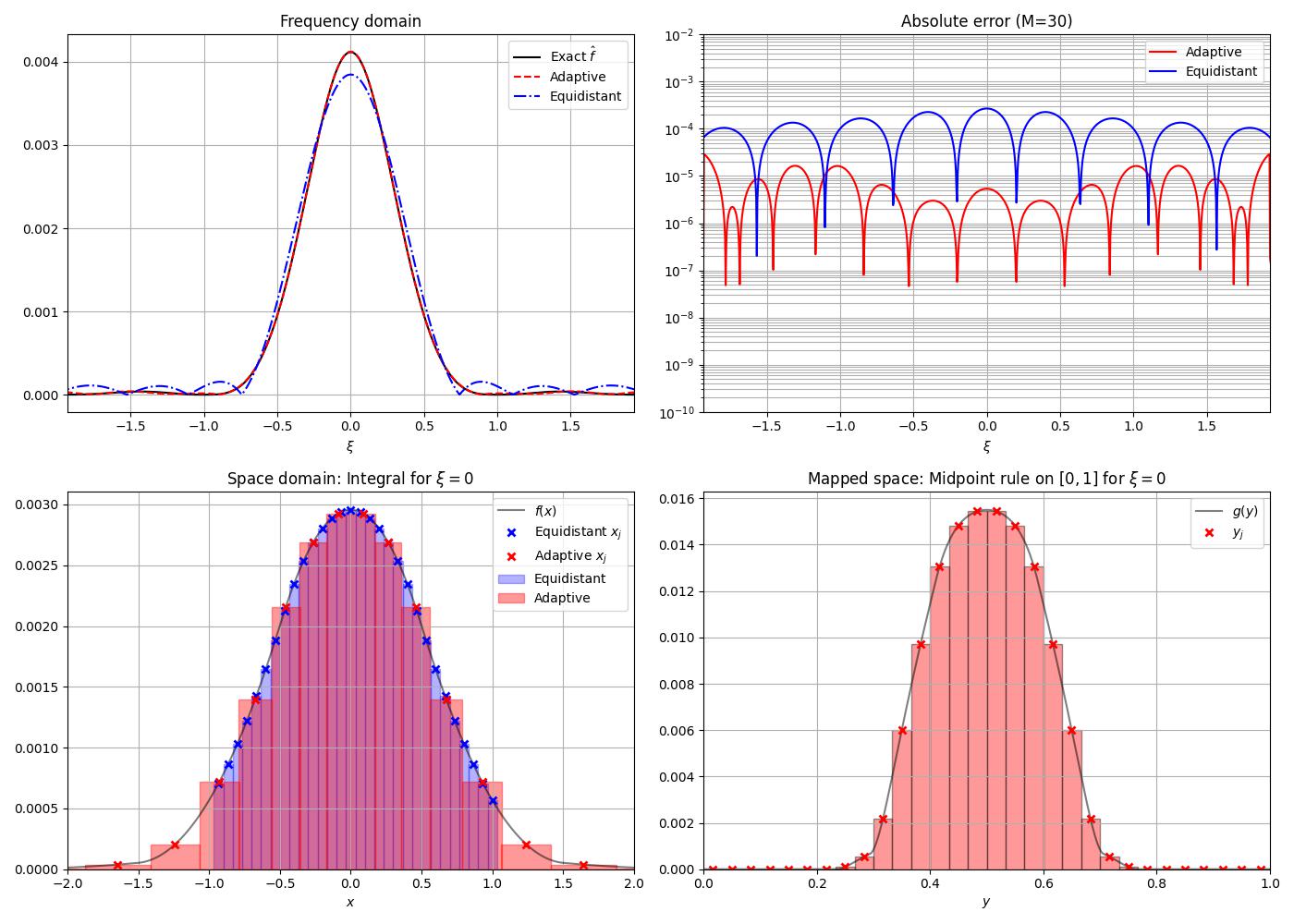}
    \caption{Approximation of the Fourier transform of the function~\eqref{eq:fab} with parameters $a=8$ and $b=3$ with equispaced points and sinc interpolation as in~\cite{EhGrKl24} (blue) and with our method (red) with the optimal choice of the distribution $\mu$ using $M=30$ samples for both methods. In the bottom left picture not all red points are shown, since some are outside of the plotted interval.}
    \label{fig:intro}
\end{figure}

Figure~\ref{fig:intro} illustrates the fundamental spatial trade-off between classical equispaced sampling from~\cite{EhGrKl24} with sinc interpolation and our proposed quadrature. The equispaced grid truncates the infinite integration domain $\R$, discarding the asymptotic tails and thereby introducing a truncation error. Conversely, our approach structurally captures the entire real line. The non-equispaced samples $\{x_j\}_{j=1}^M$ on $\R$ are mapped to $[0,1]$ via the cumulative distribution function of the density $\mu$. And then on the compact domain $[0,1]$ we use the midpoint rule to approximate the integral, see right lower picture in Figure~\ref{fig:intro}. 
While this global coverage may initially necessitate a slight reduction in central node density for small $M$, this trade-off is overcome by the nature of our mapping method. As the total number of points $M$ increases, the transformation leads to a larger density of samples directly in the central region. Consequently, the mapped scheme achieves high-resolution accuracy at the peak while preserving the infinite tails, without ever requiring truncation in the space domain.

\subsection{Our contributions}

The main contributions of this paper can be summarized as follows.
\begin{itemize}
\item \emph{A deterministic non-equispaced quadrature for the Fourier transform:} Given a probability density $\mu$ with cumulative distribution function $\Psi$, we sample $f$ at the deterministic non-equispaced nodes $x_j = \Psi^{-1}\big(\tfrac{2j-1}{2M}\big)$, $j=1,\ldots,M$, and approximate $\hat f$ by the weighted exponential sum~\eqref{eq:ftilde}. The transformation compactifies the real line, so that no truncation of the space domain is needed, and the sum can be evaluated at $N$ arbitrary frequencies by the NNFFT in $\mathcal O(K\log K + M + N)$ arithmetical operations~\cite[Section 7.3]{PlPoStTa23}.
\item \emph{Error estimates for arbitrary frequencies and in $L_p$:} We prove error bounds for every $\xi \in\R$, not only for the equispaced frequencies $\hat f\big(\tfrac kp\big)$ covered by~\cite{EhGrKl24}. And we measure the error in $L_p(\R)$ for all $p\geq 1$, see~\eqref{eq:L_p-error}. Only a truncation in the frequency domain is required. The key technical ingredients are the aliasing identity for the midpoint rule (Theorem~\ref{thm:aliasing_midpoint}) combined with stationary phase estimates in the spirit of van der Corput's lemma for the transformed oscillatory integrals (Theorem~\ref{thm:fourier_coeffs}).
\item \emph{Polynomial decay:} For functions with polynomial decay of order $a$ in space and $b$ in frequency we recover in the heavy-tail regime in the limit the error rate $M^{-\frac{(a-1/2)(b-1/2)}{a+b-1}}$ of equispaced sampling~\cite{EhGrKl24}, see Section~\ref{sec:polynomial}. The density~\eqref{eq:mu_pol} carries two additional degrees of freedom: the tail exponent $\alpha$, which we optimize explicitly in~\eqref{eq:alpha_opt}, and the variance $\sigma$, which leaves the rates untouched but can be tuned to reduce the pre-asymptotic error constants considerably, see Figure~\ref{fig:intro} and Section~\ref{sec:numerics}.
\item \emph{(Sub-)exponential decay:} For functions with exponential decay we state with our method in Section~\ref{sec:exponential} sub-exponential error decay similar to the equispaced sampling as in~\cite{EhGrKl24}. While having the additional variance parameter $\sigma$ as tuning parameter, we reduce the pre-asymptotic error constants for analytic functions.
\end{itemize}

The paper is organized as follows. In Section~\ref{sec:approach} we introduce the transformed quadrature rule, the truncated approximation~\eqref{eq:ftilde} and the aliasing identity of the midpoint rule. Section~\ref{sec:relation} embeds classical equispaced sampling into this framework and relates the resulting error splitting to the estimates of~\cite{EhGrKl24}. Section~\ref{sec:rates} contains the error analysis. After recalling van der Corput's lemma, we treat weights of polynomial growth, optimize the parameters $\alpha$ and $\sigma$ of the sampling density. In the case of (sub-)exponential weight functions, we also use a polynomial distribution to state error estimates. In Section~\ref{sec:numerics} we present numerical experiments which confirm the theoretical findings. Frequently used integral computations are collected in the appendix.

\section{Our approach: Non-equispaced sampling}\label{sec:approach}
Instead of using an equispaced grid in time and frequency domain like in~\cite{EhGrKl24}, we introduce a measure $\mu$ on the time domain. In order to avoid dealing with probability results as for random Fourier features and to be able to give error decay rates comparable to these in~\cite{EhGrKl24}, we use the cumulative distribution function to transform equispaced samples on the unit interval $[0,1]$ to $\R$, which model the distribution $\mu$ on $\R$.\\

Combining an analytic change of variables with a simple equispaced quadrature rule is a
classical idea in numerical integration. Sinc and double-exponential (tanh–sinh) rules
use a transformation together with the truncated trapezoidal rule and place the equispaced nodes so that the integrand decays rapidly at the images of the endpoints. There
the transformation typically maps a finite or semi-infinite interval onto $\R$,~\cite{TaMo74,Stenger93,MoSu01,TrWe14}.
In the context of multivariate approximation, transformations to the torus have been employed for lattice rules and wavelet methods~\cite{NuSu21,NaPo20,LiPo22}. The new difficulty in our setting is that the integrand is oscillatory with an unbounded frequency parameter $\xi$, and that we aim for error estimates which are explicit in $\xi$ in order to control the $L_p$-error on the whole real line.

\subsection{Decay in space and frequency domain}
We quantify the decay of a function by the use of a weight function $v\colon\R\rightarrow (0,\infty)$.
We introduce the following norms and the corresponding function spaces by
\begin{align*}
\norm{f}_{\F_v} &\coloneqq  \sup_{x\in \R}  |f(x)| \, v(x), & \quad \F_v &\coloneqq \{f\in C(\R)\mid \norm{f}_{\F_v}<\infty\}, \\
\norm{f}_{\F_{v,m}} &\coloneqq  \sup_{x\in \R}  \left(\sum_{\ell = 0}^m|f^{(\ell)}(x)|(1+|x|^2)^{\ell/2}\right) \, v(x), & \quad \F_{v,r} &\coloneqq \{f\in C(\R)\mid \norm{f}_{\F_{v,r}}<\infty\} .
\end{align*}

Typical examples, which we will study in this paper include polynomial weights $v(x) = (1+x^2)^{a/2}$ and exponential weights $v(x) = \e^{a|x|}$ for $a>0$.
Similarly, we characterize the decay of the Fourier transform of $f$ via $\norm{\hat f}_{\F_w}$ for a weight function $w$.

\subsection{Description of the sampling method}
Let $\mu$ be a density with cumulative distribution function $\Psi\colon \R\rightarrow [0,1]$. 
In the following we will always use the variable $x$ on $\R$ and the corresponding variable $y$ on $[0,1]$. Now, we change the integral to 
\begin{align*}
    \hat f(\xi) &= \int_{-\infty}^\infty f(x)\e^{-2\pi\im x\xi}\d x=\int_{0}^1 \frac{f(\Psi^{-1}(y))}{\mu(\Psi^{-1}(y))}\e^{-2\pi\im \Psi^{-1}(y) \xi}\d y.
\end{align*}
This integral is approximated by the midpoint rule by
\begin{align}
    \hat f(\xi) &\approx \frac{1}{M} \sum_{j=1}^M  \frac{f(\Psi^{-1}(y_j))}{\mu(\Psi^{-1}(y_j))}\e^{-2\pi\im \Psi^{-1}(y_j) \xi} = \frac{1}{M} \sum_{j=1}^M  \frac{f(x_j)}{\mu(x_j)}\e^{-2\pi\im x_j \xi} \eqqcolon \tilde f(\xi),\label{eq:f_tilde_approx}\\
    y_j &\coloneqq   \frac{(2j-1)}{2M} . \notag
\end{align}
There, the samples $x_j\in \R$ are distributed according to a discrete version of the density $\mu$.
We note that we compute $\tilde f(\xi_\ell)$ for $\ell =0,\ldots, N-1$ with $|x_j\xi_\ell|<K$ by applying the nonequispaced Fourier transform in space and frequency domain (NNFFT)
to the vector $\left( \frac{f(x_j)}{\mu(x_j)}\right)_{j=1}^M$ in ${\cal O}(K\log K +M+N)$ arithmetical operations, see \cite[Section 7.3]{PlPoStTa23}.
We restrict the evaluation of the Fourier transform $\hat f$ to a compact interval $[-\Omega,\Omega]$ and set the approximation $\tilde f$ to zero outside of this interval,
\begin{equation}\label{eq:ftilde}
\tilde f(\xi) \coloneqq \begin{cases}
\frac{1}{M} \sum_{j=1}^M  \frac{f(x_j)}{\mu(x_j)}\e^{-2\pi\im\xi x_j } &\text{ if }|\xi|\leq \Omega\\
0& \text{ otherwise.}\\   
\end{cases}
\end{equation}
In this paper we will state error decay results for the $L_p$-error for $p\geq 1$, defined by
\begin{equation}\label{eq:L_p-error}
E_p \coloneqq \left( \int_{\R}|\hat f(\xi) - \tilde f(\xi)|^p \d \xi \right)^{1/p}.
\end{equation}
By truncating in frequency domain, we receive the following estimate, 
\begin{align}
E_p^p&=\int_{\R}|\hat f(\xi) - \tilde f(\xi)|^p \d \xi \leq \int_{-\infty}^{-\Omega}|\hat f(\xi)|^p \d \xi +\int_{\Omega}^{\infty}|\hat f(\xi)|^p \d \xi +  \int_{-\Omega}^\Omega|\hat f(\xi) - \tilde f(\xi)|^p \d \xi, \notag\\
&\leq \norm{\hat f}_{\F_{w}}^p \left(\int_{-\infty}^{-\Omega} \frac{1}{|w(\xi)|^p} \d \xi + \int_{\Omega}^\infty \frac{1}{|w(\xi)|^p} \d \xi \right) +  \int_{-\Omega}^\Omega|\hat f(\xi) - \tilde f(\xi)|^p \d \xi. \label{eq:Wabfall}
\end{align}\par
Some truncation in the frequency domain as in~\eqref{eq:ftilde} is unavoidable. The exponential sum in~\eqref{eq:f_tilde_approx} is a finite linear combination of the functions $\xi\mapsto\e^{-2\pi\im x_j\xi}$ and therefore an almost periodic function of $\xi$, see e.g.~\cite{Co09,DuKo21}. In particular, it does not decay for $|\xi|\rightarrow\infty$, whereas $\hat f$ does. 
The error of the midpoint rule in \eqref{eq:f_tilde_approx} depends on the smoothness of the underlying function
\begin{equation}\label{eq:g}
g(y)\coloneqq \frac{f(\Psi^{-1}(y))}{\mu(\Psi^{-1}(y))}\e^{-2\pi\im \Psi^{-1}(y) \xi}.
\end{equation}
This function $g$ is periodic on the interval $[0,1]$ if $f\in \F_v$ and
$$\lim_{x\rightarrow \pm \infty} v(x)\mu(x)=\infty, $$
since in this case
$$g(1)=g(0) = \lim_{x\rightarrow \pm \infty} \frac{f(x)}{\mu(x)}\e^{-2\pi\im x} = 0.$$
For the function $g\colon[0,1]\rightarrow\C$, we denote the Fourier coefficients by
$$c_k(g) \coloneqq \int_0^1 g(y) \mathrm{e}^{-2\pi\mathrm{i} k y} \,\mathrm{d}y.$$
The following theorem states an error estimate of the midpoint rule in \eqref{eq:f_tilde_approx} using the decay of the Fourier coefficients of $g$.
\begin{theorem}\label{thm:aliasing_midpoint}
Let $g\colon [0,1) \to \mathbb{C}$ be a function whose Fourier series converges absolutely. Let the exact integral of $g$ over one period be $c_0(g)$. The numerical approximation by the midpoint rule with $M \in \mathbb{N}$ equispaced subintervals is given by
\begin{equation*}
    I_M(g) \coloneqq \frac{1}{M} \sum_{j=1}^M g\left( \frac{2j-1}{2M} \right).
\end{equation*}
Then, the quadrature error satisfies the identity
\begin{equation*}
    I_M(g) -c_0(g)= \sum_{k \in \mathbb{Z} \setminus \{0\}} (-1)^{k} c_{k M}(g).
\end{equation*}
\end{theorem}
\begin{proof}
Since the Fourier series of $g$ converges absolutely, we can represent $g$ pointwise by its Fourier series
\begin{equation*}\label{eq:FRg}
    g(y) = \sum_{k \in \mathbb{Z}} c_k(g) \e^{2\pi\im k y}.
\end{equation*}
Substituting this representation into the quadrature formula ${I}_M(g)$ yields
\begin{align*}
    I_M(g) &= \frac{1}{M} \sum_{j=1}^M \sum_{k \in \mathbb{Z}} c_k(g) \e^{2\pi\im k \frac{2j-1}{2M}}
     = \sum_{k \in \mathbb{Z}} c_k(g)\, \e^{-\frac{\pi\im k}{M}} \left( \frac{1}{M} \sum_{j=1}^M \e^{2\pi\im \frac{k\,j}{M} } \right).
\end{align*}
The inner sum over $j$ is a finite geometric series representing the sum of the $M$-th roots of unity. By the orthogonality properties of the roots of unity, this sum equals $M$ if $k$ is an integer multiple of $M$, and it is $0$ otherwise. 
Therefore, the outer sum collapses strictly to
\begin{align*}
    I_M(g) &= \sum_{k \in \mathbb{Z}} c_{kM}(g) \e^{-\pi\im k}= \sum_{k \in \mathbb{Z}} (-1)^kc_{k M}(g).
\end{align*}
Separating the term for $k = 0$ and noting that the zeroth Fourier coefficient corresponds to the exact integral, yields
\begin{equation*}
    I_M(g) = c_0(g) + \sum_{k\in \mathbb Z\setminus \{0\} } (-1)^k c_{k M}(g).
\end{equation*}
Subtracting $c_0(g)$ concludes the proof.
\end{proof}
For all trigonometric monomials $g(y)=\e^{2\pi \im ny}$ with a frequency $n\in \N$ that is not a multiple of $M$, this quadrature is exact. For a $1$-periodic function $g\in C^m$, the Fourier coefficients satisfy $|c_k(g)|=\mathcal O(|k|^{-m})$, see e.g.~\cite{PlPoStTa23}, such that the quadrature error decays at least at the rate $\mathcal O(M^{-m})$.\par
In order to apply the previous theorem to the function \eqref{eq:g} in our main approximation~\eqref{eq:f_tilde_approx} of the Fourier transform, we have to calculate the decay of the Fourier coefficients of the function $g$ for fixed $\xi$ with $|\xi|\le \Omega$, which we will do in Section~\ref{sec:rates} for different settings.

\subsection{Relation of our approach to equispaced sampling}\label{sec:relation}
To state the relation between our approach and the estimates in the equispaced case from~\cite{EhGrKl24}, we first note how equispaced sampling fits into our setting. It corresponds to the choice of the uniform density on an interval $[-\tfrac P2, \tfrac P2]$, which already contains a truncation of the space domain. In this case,
$$\mu(x)=\begin{cases}\frac 1P &\text{ if } -\tfrac P2 \leq x\leq \tfrac P2,\\
0 &\text{ otherwise},\end{cases}  \qquad \Psi(x) =\begin{cases}
0 &\text{ if } x<-\tfrac P2,\\
\frac{x}{P}+\tfrac{1}{2} &\text{ if }-\tfrac P2\leq x\leq \tfrac P2,\\
1 &\text{ if } x>\tfrac P2,\end{cases} $$
such that $\Psi^{-1}(y) = P\left(y-\tfrac 12\right)$ and the mapped integrand~\eqref{eq:g} becomes 
\[g(y) = P\, f\!\left(P\left(y-\tfrac 12\right)\right)\e^{-2\pi\im \xi P(y-1/2)}.\]
The midpoint nodes $y_j$ from~\eqref{eq:f_tilde_approx} are mapped to the equispaced points $x_j = P\,\tfrac{2j-1-M}{2M}$ with constant step size $h=\tfrac PM$, and
$$I_M(g) = h\sum_{j=1}^M f(x_j)\,\e^{-2\pi\im \xi x_j}$$
is exactly the Riemann sum in~\eqref{eq:f_hat_approx_equi}, up to a shift by half a step size. In contrast to a density with full support $\R$, we now have $c_0(g)\neq \hat f(\xi)$, but rather
$$c_0(g)=\int_{-\tfrac P2}^{\tfrac P2} f(x)\,\e^{-2\pi\im x\xi}\d x = \widehat{f_P}(\xi), \qquad f_P \coloneqq f\cdot {\bm 1}_{[-\frac P2,\frac P2]},$$ and
\begin{align*}
|I_M(g)-\hat f(\xi)| &=  |I_M(g)- c_0(g) +  c_0(g) -\hat f(\xi)  |\\
&\leq \underbrace{|I_M(g)- c_0(g)|}_{\text{discretization (aliasing) error}} + \underbrace{| c_0(g) -\hat f(\xi)  |}_{\text{truncation in space domain}}.
\end{align*}
These are precisely the two error sources considered in~\cite{EhGrKl24}.
The Fourier coefficients of the mapped integrand can be computed explicitly. Substituting $y=\Psi(x)$, we obtain
\begin{align*}
    c_{kM}(g) &= \int_0^1 g(y)\, \e^{-2\pi\im kMy}\d y
    = \int_{-P/2}^{P/2} f(x)\, \e^{-2\pi\im \xi x}\, \e^{-2\pi\im kM \left(\frac{x}{P} + \frac{1}{2}\right)} \,\d x\\
    &= (-1)^{kM}\, \widehat{f_P}\left(\xi +\frac{kM}{P}\right).
\end{align*}
Inserting this back into Theorem \ref{thm:aliasing_midpoint} yields
\begin{align*}
    I_M(g) - c_0(g) &= \sum_{k \in \mathbb{Z} \setminus \{0\}} (-1)^{k(M+1)}\, \widehat{f_P}\left(\xi+\frac{k}{h}\right),
\end{align*}
where we used $\tfrac{kM}{P}=\tfrac kh$. Hence, for equispaced sampling the discretization error is a pure aliasing sum. It collects the values of $\widehat{f_P}$ on the shifted grid $\xi + \tfrac 1h\Z$, and the alternating sign stems from the half-step offset of the midpoint nodes. This is the Poisson summation formula in disguise, see e.g.~\cite{PlPoStTa23}, and it explains the role of the two parameters in~\eqref{eq:f_hat_approx_equi}. The truncation parameter $P$ controls the truncation error $|\widehat{f_P}(\xi)-\hat f(\xi)|$ via the spatial decay of $f$, whereas the step size $h$ controls the aliasing error via the decay of $\hat f$.\par

In addition to~\cite{EhGrKl24}, a uniform error of the approximation with uniform samples was studied in~\cite{PoTa25}. This then allows the evaluation at any point within the interval.
\section{Error decay rates for non-equispaced sampling}\label{sec:rates}
In this section we present our main results. In Section~\ref{sec:polynomial} we study the case where the function $f$ and the Fourier transform $\hat f$ have polynomial decay and state error decay rates in Theorem~\ref{thm:error_hatf}. In Section~\ref{sec:exponential} we study the case of (sub-)exponential decay, where we state final bounds for the $L_p$-error in Corollary~\ref{cor:error_decay_exp}.

\subsection{Van der Corput's Lemma}
Here we state some result from the literature, which we will use later.
In the study of oscillatory integrals, a well-known result which is of fundamental importance is the van der Corput’s lemma,~\cite[Chapter VIII, Prop. 2]{Stein93}.

\begin{lemma}[Van der Corput's Lemma]
\label{lem:van_der_corput}
Suppose $\phi$ is real-valued and smooth in $[a, b]$ and that $|\phi^{(k)}(x)|\geq 1$ for all $x\in [a,b]$. Suppose that $k\geq 2$. Then 
$$\left|\int_{a}^b \e^{\im \lambda \phi(x)} \d x\right|\leq c_{k} \lambda^{-1/k},$$
where the constant $c_k$ is independent of $\phi$ and $\lambda$. If additionally $\psi\colon[a,b]\rightarrow\C$ is continuously differentiable, then
$$\left|\int_{a}^b \psi(x)\e^{\im \lambda \phi(x)} \d x\right|\leq c_{k} \lambda^{-1/k}\left(\norm{\psi}_{L_\infty([a,b])} + \norm{\psi'}_{L_1([a,b])}\right).$$
\end{lemma}
Part one of the above lemma is what is classically called van der Corput’s
lemma. Part two (which follows from the first part by integration by parts), is the version one usually finds convenient to use in many applications.
\subsection{Polynomial decay in time and frequency}\label{sec:polynomial}
We discuss here the special case of functions of polynomial spatial decay and also polynomial spectral regularity. In contrast to~\cite{EhGrKl24}, where the polynomial weights $(1+|x|)^a$ are used, we use here $v_a(x) = (1+x^2)^{a/2}$, which have the same asymptotical decay, but we need the additional smoothness at the midpoint $x=0$. Finite smoothness is often replaced by polynomial decay of the Fourier transform. If the derivatives $f^{(\ell)}$ are in $L_1(\R)$ for $\ell = 0,\ldots, b$ with $b\in \N$, then it follows that $\sup_{\xi\in \R} |\hat f(\xi)|(1+\xi^2)^{b/2}<\infty$.\par

Let $v(x) = (1+x^2)^{a/2}$ and $w(\xi) = (1+\xi^2)^{b/2}$ be weight functions. Let $f\in \F_{v,m}  $ with $m\in \{1,2,\ldots, b-1\}$ and $\hat f\in \F_{w}$. 
We choose the polynomial distribution 
\begin{equation}\label{eq:mu_pol}
\mu(x) = C_{\alpha,\sigma} \left(1+\frac{x^2}{\sigma^2}\right)^{-\alpha/2}, \quad C_{\alpha,\sigma} = \frac{1}{\sigma} \frac{\Gamma(\tfrac \alpha 2)}{\sqrt{\pi}\,\Gamma(\tfrac{\alpha -1}{2})},
\end{equation}
and the parameters $\alpha, \sigma >0$ have to be chosen appropriately depending on spatial decay $a$ and spectral regularity $b$.
Sometimes we will omit the index $\sigma$ in the constant $C_{\alpha,\sigma}$, if the variance equals $\sigma=1$.
In the following we analyze the error terms and study how this parameter $\alpha$ has to be chosen to balance the error terms.
As a preliminary result we calculate the derivatives of the function $\mu$.

\begin{lemma}
For polynomial density $\mu$ defined in~\eqref{eq:mu_pol} with $\sigma=1$ and $\ell \in \mathbb{N}_0$ and $x \in \mathbb{R}$, the derivatives of $\mu$ are bounded by
\begin{equation}\label{eq:ratio_mu}
\left| \frac{\mu^{(\ell)}(x)}{\mu(x)} \right|
\leq \frac{\Gamma(\alpha+\ell)}{\Gamma(\alpha)} (1+x^2)^{-\ell/2}.
\end{equation}
\end{lemma}
\begin{proof}
We construct the Taylor series of $\mu$ by
\begin{align*}
\mu(x-h) &= C_{\alpha,1} \left(1+(x-h)^2\right)^{-\alpha/2} =  C_{\alpha,1}\left(1+x^2-2xh+h^2\right)^{-\alpha/2} \\
&= C_{\alpha,1}\left(\frac{1}{1+x^2}\right)^{\alpha/2}\left(1-\frac{2xh}{1+x^2}+\frac{h^2}{1+x^2}\right)^{-\alpha/2} .
\end{align*}
We substitute $y = \frac{x}{\sqrt{1+x^2}}$ and $t= \frac{h}{\sqrt{1+x^2}}$. Then, we use the fact that the Gegenbauer polynomials can be described in terms of their generating function,
$$\left(\frac{1}{1-2yt+t^2}\right)^{\alpha/2}= \sum_{\ell=0}^\infty C_{\ell}^{(\alpha/2)}(y)t^\ell.$$
This yields
\begin{align*}
\mu(x-h) &=C_{\alpha,1} \left(\frac{1}{1+x^2}\right)^{\alpha/2}\sum_{\ell=0}^\infty C_{\ell}^{(\alpha/2)}(y)t^\ell\\
&=C_{\alpha,1}\sum_{\ell=0}^\infty  \left(\frac{1}{1+x^2}\right)^{\alpha/2+\ell/2} C_{\ell}^{(\alpha/2)}\left(\frac{x}{\sqrt{1+x^2}}\right)h^\ell.
\end{align*}
The generating function converges for $|t|<1$, i.e., for $|h|<\sqrt{1+x^2}$, which is precisely the radius of convergence of the Taylor series of $\mu$ around $x$, so that comparing the coefficients before $h^\ell$ with the classical formula for the Taylor series leads to
$$\frac{(-1)^\ell}{\ell!}\mu^{(\ell)}(x) = C_{\alpha,1}\left(\frac{1}{1+x^2}\right)^{\alpha/2+\ell/2} C_{\ell}^{(\alpha/2)}\left(\frac{x}{\sqrt{1+x^2}}\right).$$
 Finally,
$$\left|\frac{\mu^{(\ell)}(x)}{\mu(x)}\right| =\ell!\left(\frac{1}{1+x^2}\right)^{\ell/2}\left| C_{\ell}^{(\alpha/2)}\left(\frac{x}{\sqrt{1+x^2}}\right) \right|.$$
Since the Gegenbauer polynomials with parameter $\tfrac{\alpha}{2}>0$ attain their maximal absolute value on $[-1,1]$ at the endpoints, see~\cite[Theorem~7.33.1]{Sz75}, they are bounded by
$$\left|C_\ell^{(\alpha/2)}(y) \right|\leq C_\ell^{(\alpha/2)}(1) = \frac{\Gamma(\ell+\alpha)}{\ell!\, \Gamma(\alpha)}.$$
and the result follows.
\end{proof}

\begin{theorem}
\label{thm:fourier_coeffs}
Let $v(x) = (1+x^2)^{a/2}$ and $w(\xi) = (1+\xi^2)^{b/2}$ be weight functions with spatial decay $a>1$ and spectral regularity $b>1$. Let $f$ be a complex-valued function such that $f\in \mathcal F_{v,m}$ and $\hat f\in \F_{w}$ with $m\in \{1,2,\ldots, b-1\}$.
Let $\mu$ be a polynomial probability density function defined in~\eqref{eq:mu_pol} with $\alpha>1$ and $\sigma=1$, and define its cumulative distribution function as $\Psi(x) = \int_{-\infty}^x \mu(t) \,\mathrm{d}t$. Fix $\xi\in [-\Omega,\Omega]$. Furthermore, let
\begin{equation}\label{eq:cond_alpha}
\alpha <\frac{a-1}{m}+1
\end{equation}
and $g$ be as defined in~\eqref{eq:g}.
In the resonant case $0<-\tfrac{\xi}{k}\leq \mu(0)$, let $x_0> 0$ be determined by $\mu(x_0)=-\tfrac \xi k$, and set
\begin{equation*}
\delta \coloneqq \min\left\{\tfrac 12,\tfrac{x_0}{4}\right\}, \qquad
D_\delta \coloneqq \left\{x\in\R \colon |x-x_0|\geq \delta \text{ and } |x+x_0|\geq \delta\right\}, \qquad
\Lambda \coloneqq \frac{1+x_0^2}{x_0\,\delta}.
\end{equation*}
Then, the decay of the Fourier coefficients $c_k(g)$ is bounded by
\begin{small}
    \begin{equation*}
        |c_k(g)| \lesssim_{\alpha,m}  \begin{cases}
        \norm{f}_{\mathcal F_{v,m}} \displaystyle\int_{-\infty}^\infty  \frac{1}{(1+x^2)^{(a+m)/2}\,|\xi+ k\mu(x)|^m} \d x & \hspace*{-35pt}\text{if } \tfrac{\xi}{k}\geq 0 \text{ or } -\tfrac{\xi}{k}\geq 2\mu(0), \\[2ex]
        \Lambda^m \norm{f}_{\mathcal F_{v,m}} \displaystyle\int_{D_\delta}  \frac{1}{(1+x^2)^{(a+m)/2}\,|\xi+ k\mu(x)|^m} \d x + \frac{\norm{f}_{\mathcal F_{v,1}} }{v(x_0)\sqrt{|k \mu'(x_0)|}} &\text{if } 0< -\tfrac{\xi}{k}\leq \mu(0).
        \end{cases}
    \end{equation*}%
    \end{small}%
The remaining transitional regime $\mu(0)<-\tfrac{\xi}{k}<2\mu(0)$ is discussed in Remark~\ref{rem:resonant}.
\end{theorem}
\begin{proof}
We evaluate the Fourier coefficients by applying the substitution $y = \Psi(x)$, which implies $x = \Psi^{-1}(y)$ and the differential relation $\mathrm{d}y = \mu(x) \,\mathrm{d}x$. Substituting this into the definition of $c_k(g)$ yields
\begin{align}\label{eq:c_k_g}
    c_k(g) &= \int_{-\infty}^\infty \frac{f(x)}{\mu(x)} \e^{-2\pi\im \xi x} \e^{-2\pi\im k \Psi(x)} \mu(x) \,\mathrm{d}x =\int_{-\infty}^\infty f(x) \e^{-2\pi\im (\xi x + k \Psi(x))}  \mathrm{d}x\notag\\
    &= \int_{-\infty}^\infty f(x) \e^{-2\pi\im \Theta_{k,\xi}(x)}  \mathrm{d}x ,
\end{align}
where we define $\Theta_{k,\xi}(x) \coloneqq \xi x + k \Psi(x)$. The derivatives of the phase are
\begin{equation*}
    \Theta_{k,\xi}'(x) = \xi + k \mu(x), \quad \text{and} \quad \Theta_{k,\xi}^{(r)}(x) = k \mu^{(r-1)}(x) \quad \text{ for }r\geq 2. 
\end{equation*}
Now we distinguish whether the phase $\Theta_{k,\xi}$ has a stationary point on $\R$ or not. Since $\mu$ takes its values in $(0,\mu(0)]$, the derivative $\Theta_{k,\xi}'(x)=\xi+k\mu(x)$ vanishes for some $x\in \R$ if and only if $\xi$ and $k$ have opposite signs and $0<-\tfrac{\xi}{k}\leq\mu(0)$.
\par
\paragraph{Case 1: $\tfrac{\xi}{k}\geq0$ or $-\tfrac{\xi}{k}\geq 2\mu(0)$.} In this case $\Theta_{k,\xi}'$ does not vanish, and moreover
\begin{equation}\label{eq:case1_lower}
|\Theta_{k,\xi}'(x)|=|\xi+k\mu(x)|\geq \tfrac14\big(|\xi|+|k|\mu(x)\big)\geq \tfrac14|k|\mu(x),\qquad x\in\R.
\end{equation}
Indeed, for $\tfrac{\xi}{k}\geq0$ we even have $|\xi+k\mu(x)|=|\xi|+|k|\mu(x)$, while for $-\tfrac{\xi}{k}\geq2\mu(0)$ we have $|k|\mu(x)\leq|k|\mu(0)\leq\tfrac12|\xi|$, hence $|\xi+k\mu(x)|\geq|\xi|-|k|\mu(x)\geq\tfrac12|\xi|\geq\tfrac14(|\xi|+|k|\mu(x))$. In the following we carry out the estimates for $\tfrac{\xi}{k}\geq0$, where $|\xi+k\mu(x)|=|\xi|+|k|\mu(x)$. The subcase $-\tfrac{\xi}{k}\geq2\mu(0)$ is identical up to the factor $4^m$ from~\eqref{eq:case1_lower}, which is absorbed into $\lesssim_{\alpha,m}$. Thus, we may apply integration by parts using the identity $\e^{-2\pi\im \Theta_{k,\xi}} = (-2\pi\im \Theta_{k,\xi}')^{-1} \frac{\mathrm{d}}{\mathrm{d}x} \e^{-2\pi\im \Theta_{k,\xi}}$,
\begin{equation}\label{eq:c_k1}
    c_k(g) = \frac{1}{2\pi\im} \int_{-\infty}^\infty \frac{\mathrm{d}}{\mathrm{d}x} \left( \frac{f(x)}{\Theta'_{k,\xi}(x)} \right) \e^{-2\pi\im \Theta_{k,\xi}(x)} \,\mathrm{d}x.
\end{equation}
The boundary terms vanish since $\lim_{|x| \to \infty} f(x) = 0$. 
For the case where $m=1$, we receive by using the estimate \eqref{eq:ratio_mu},
\begin{align*}
    |c_k(g)| &= \left|\frac{1}{2\pi\im} \int_{-\infty}^\infty \left( \frac{f'(x)}{\Theta'_{k,\xi}(x) }- \frac{f(x)\Theta''_{k,\xi}(x)}{(\Theta'_{k,\xi}(x))^2} \right) \e^{-2\pi\im \Theta_{k,\xi}(x)} \d x\right|\\
    &\leq \frac{1}{2\pi} \int_{-\infty}^\infty \left|\frac{f'(x)}{\Theta'_{k,\xi}(x) }- \frac{f(x)\Theta''_{k,\xi}(x)}{(\Theta'_{k,\xi}(x))^2} \right|  \d x\\
    &\leq \frac{1}{2\pi} \int_{-\infty}^\infty \left|\frac{f'(x)}{\xi+k\mu(x) }\right|+\left| \frac{f(x)k \mu'(x)}{(\xi+k\mu(x) )^2} \right|  \d x\\
    &\leq  \frac{1}{2\pi} \int_{-\infty}^\infty \left|\frac{f'(x)}{\xi+k\mu(x) }\right|+\left| \frac{\alpha f(x)}{(1+|x|^2)^{1/2}(\xi+k\mu(x) )} \right|  \d x\\
    &\leq \frac{\norm{f}_{v,1}(1+\alpha)}{2\pi} \int_{-\infty}^\infty \left| \frac{1}{v(x)(1+|x|^2)^{1/2}(\xi+k\mu(x) )} \right|  \d x.
\end{align*}

For $m=2$ we use integration by parts again in~\eqref{eq:c_k1}. The boundary terms vanish since additionally $\lim_{|x| \to \infty} f'(x) = 0$, such that,
\begin{small}
\begin{align*}
    c_k(g) &= \frac{1}{(2\pi\im)^2} \int_{-\infty}^\infty  \frac{\mathrm{d}}{\mathrm{d}x}\left(\frac{\mathrm{d}}{\mathrm{d}x} \left( \frac{f(x)}{\Theta'_{k,\xi}(x)} \right)\frac{1}{\Theta'_{k,\xi}(x)}\right) \e^{-2\pi\im \Theta_{k,\xi}(x)} \,\mathrm{d}x\\
    &=\frac{1}{(2\pi\im)^2} \int_{-\infty}^\infty  \frac{\mathrm{d}}{\mathrm{d}x}\left(  \frac{f'(x)\Theta'_{k,\xi}(x)-f(x)\Theta{''}_{k,\xi}(x)}{(\Theta'_{k,\xi}(x))^2} \frac{1}{\Theta'_{k,\xi}(x)} \right)\e^{-2\pi\im \Theta_{k,\xi}(x)} \,\mathrm{d}x\\
     &=\frac{1}{(2\pi\im)^2} \int_{-\infty}^\infty  \frac{\mathrm{d}}{\mathrm{d}x}\left(  \frac{f'(x)}{(\Theta'_{k,\xi}(x))^2}-\frac{f(x)\Theta{''}_{k,\xi}(x)}{(\Theta'_{k,\xi}(x))^3} \right) \e^{-2\pi\im \Theta_{k,\xi}(x)} \,\mathrm{d}x\\
    &=\frac{1}{(2\pi\im)^2} \int_{-\infty}^\infty  \left(  \frac{f''(x)(\Theta'_{k,\xi}(x))^2-2f'(x)(\Theta'_{k,\xi}(x))\Theta{''}_{k,\xi}(x)}{(\Theta'_{k,\xi}(x))^4}-\frac{(f'(x)\Theta{''}_{k,\xi}(x)+f(x)\Theta{'''}_{k,\xi}(x)) }{(\Theta'_{k,\xi}(x))^3}\right. \\
    &\quad +\left.\frac{3f(x)(\Theta{''}_{k,\xi}(x))^2(\Theta'_{k,\xi}(x))^2}{(\Theta'_{k,\xi}(x))^6} \right) \e^{-2\pi\im \Theta_{k,\xi}(x)} \,\mathrm{d}x\\
     &=\frac{1}{(2\pi\im)^2} \int_{-\infty}^\infty  \left(  \frac{f''(x)}{(\Theta'_{k,\xi}(x))^2}-\frac{3f'(x)\Theta{''}_{k,\xi}(x)}{(\Theta'_{k,\xi}(x))^3}-\frac{f(x)\Theta{'''}_{k,\xi}(x)) }{(\Theta'_{k,\xi}(x))^3} +\frac{3f(x)(\Theta{''}_{k,\xi}(x))^2}{(\Theta'_{k,\xi}(x))^4} \right) \e^{-2\pi\im \Theta_{k,\xi}(x)} \,\mathrm{d}x.
\end{align*}
\end{small}
Taking the absolute value inside the integral yields due to~\eqref{eq:ratio_mu},
\begin{small}
\begin{align*}
    |c_k(g)| &\leq\frac{1}{(2\pi)^2} \int_{-\infty}^\infty  \left| \frac{f''(x)}{(\xi+k\mu(x))^2}\right|+\left|\frac{3f'(x)\mu'(x)}{(\xi+k\mu(x))^2\mu(x)}\right|+\left|\frac{f(x)\mu''(x) }{(\xi+ k\mu(x))^2\mu(x)}\right| +\left|\frac{3f(x)(\mu'(x))^2}{(\xi+k\mu(x))^2\mu^2(x)} \right|\mathrm{d}x\\
    &\leq \frac{1+ 3\alpha+\alpha(\alpha+1) + 3\alpha^2}{(2\pi)^2}\norm{f}_{\mathcal F_{v,2}} \int_{-\infty}^\infty  \frac{1}{v(x)(1+|x|^2)(\xi+ k\mu(x))^2} \d x\\
       &= \frac{4\alpha^2 +4\alpha +1}{(2\pi)^2}\norm{f}_{\mathcal F_{v,2}} \int_{-\infty}^\infty  \frac{1}{v(x)(1+|x|^2)(\xi+ k\mu(x))^2} \d x.
\end{align*}
\end{small}
For $m>2$, we have to use integration by parts again. We omit the explicit calculations for this case here, but they can easily be generalized from the case where $m=2$. In all cases, $|\xi+k\mu(x)|\geq|k|\mu(x)$ and the condition~\eqref{eq:cond_alpha} ensure that the integral is finite, since the integrand decays like $(1+x^2)^{-(a+m)/2+\alpha m/2}$ and $a/2+m/2-\alpha m/2>1/2$.

\paragraph{Case 2: $0<-\tfrac{\xi}{k}\leq\mu(0)$.} In this case there are points $x_0\in \R$, such that $\Theta_{k,\xi}'(x_0) = 0$, namely $\mu(x_0) = -\frac{\xi}{k}$. For the polynomial distribution $\mu$ from~\eqref{eq:mu_pol} there are two such points $\pm x_0$ with $x_0\geq0$, which coincide only in the boundary case $-\tfrac{\xi}{k}=\mu(0)$, where $x_0=0$ and $\mu'(x_0)=0$. This degenerate case is a technical detail which we omit for the sake of readability, since it does not change the error behavior. There, one has to increase the order in van der Corput's lemma by one. In the following we thus assume $x_0>0$ and set $\delta = \min\{\tfrac12,\tfrac{x_0}{4}\}$, such that $[x_0-2\delta,x_0+2\delta]\subset[\tfrac{x_0}{2},\tfrac{3x_0}{2}]$ stays bounded away from the origin.\par

To isolate the singularities without introducing non-vanishing boundary terms from a harsh domain truncation, we employ a smooth partition of unity. We introduce a smooth bump function $\chi_{\text{sin}} \in C_c^\infty(\mathbb{R})$ centered at the singularity $x_0$ (for two singularity points, we just add the two bump functions) with a support size of $\delta > 0$, satisfying
\begin{equation*}
\chi_{\text{sin}}(x) = 
\begin{cases} 
1, & \text{for } |x - x_0| \le \delta, \\ 
0, & \text{for } |x - x_0| \ge 2\delta.
\end{cases}
\end{equation*}
We define the corresponding regular counterpart as $\chi_{\text{reg}}(x) \coloneqq 1 - \chi_{\text{sin}}(x)$. In the following we will describe the procedure for one singular point $x_0>0$, for the other one, the procedure is similar, since $\mu$ and $v$ are symmetric. An illustration of the functions can be found in Figure~\ref{fig:skizze}.
By linearity, the Fourier coefficient~\eqref{eq:c_k_g} splits into a singular and a regular integral,
\begin{equation*}
c_k(g) = \int_{-\infty}^\infty f(x) \chi_{\text{sin}}(x) \e^{-2\pi\im \Theta_{k,\xi}(x)} \,\d x + \int_{-\infty}^\infty f(x) \chi_{\text{reg}}(x) \e^{-2\pi\im \Theta_{k,\xi}(x)} \,\d x \eqqcolon I_{\text{sin}} + I_{\text{reg}}.
\end{equation*}
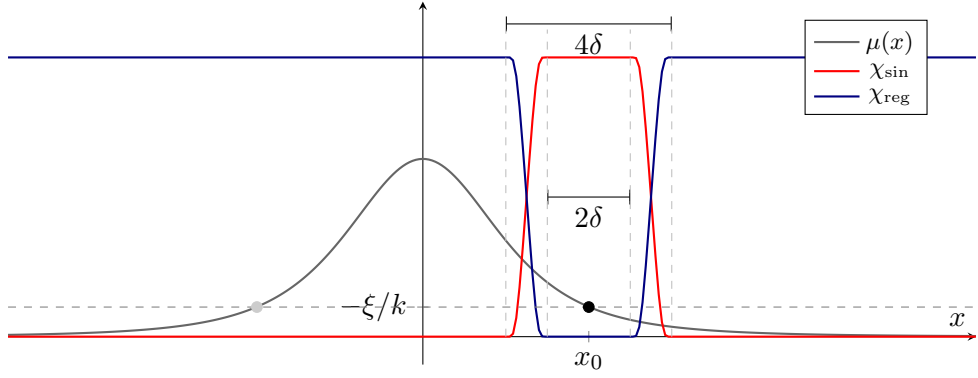
\begin{figure}[tb]
\centering
\begin{tikzpicture}
\begin{axis}[
width=0.8\linewidth,height = 0.3\textwidth, scale only axis =true,
axis lines = middle,
xlabel = {$x$},
xmin = -3, xmax = 4,
ymin = -0.1, ymax = 1.2,
domain = -3:4,
    samples = 300, 
    legend style={at={(0.95,0.95)}, anchor=north east, font=\footnotesize},
    xtick = {1.2}, 
    xticklabels = {$x_0$},
    ytick = {0.106},
    yticklabels = {$-\xi/k$},
]

\def\deltaVal{0.3}
\def\xO{1.2}

\pgfmathdeclarefunction{smoothstep}{1}{%
  \pgfmathparse{
    #1 <= 0 ? 0 : (
      #1 >= 1 ? 1 : (
        exp(-1/(#1)) / (exp(-1/(#1)) + exp(-1/(1-(#1))))
      )
    )
  }%
}

\addplot [thick, gray!80!black] {2/pi*(1+x^2)^(-2)};
\addlegendentry{$\mu(x)$}

\addplot [thick, red, domain=-3:4] {
    smoothstep(1 - (abs(x-\xO) - \deltaVal)/\deltaVal)
};
\addlegendentry{$\chi_{\text{sin}}$}

\addplot [thick, blue!50!black, domain=-3:4] {
    1 - smoothstep(1 - (abs(x-\xO) - \deltaVal)/\deltaVal)
};
\addlegendentry{$\chi_{\text{reg}}$}

\addplot [gray, thin, dashed] {0.106};
\addplot [only marks, gray!40, mark=*] coordinates {(-1.2, 0.106)};
\addplot [only marks, black, mark=*] coordinates {(\xO, 0.106)};

\draw [|-|] (axis cs: \xO-\deltaVal, 0.5) -- (axis cs: \xO+\deltaVal, 0.5) 
    node[midway, below] { $2\delta$};
\draw [|-|] (axis cs: \xO-2*\deltaVal, 1.12) -- (axis cs: \xO+2*\deltaVal, 1.12) 
    node[midway, below] {$4\delta$};

\draw [dashed, gray!50] (axis cs:\xO-\deltaVal, 0) -- (axis cs:\xO-\deltaVal, 1.1);
\draw [dashed, gray!50] (axis cs:\xO+\deltaVal, 0) -- (axis cs:\xO+\deltaVal, 1.1);
\draw [dashed, gray!50] (axis cs:\xO-2*\deltaVal, 0) -- (axis cs:\xO-2*\deltaVal, 1.1);
\draw [dashed, gray!50] (axis cs:\xO+2*\deltaVal, 0) -- (axis cs:\xO+2*\deltaVal, 1.1);

\end{axis}
\end{tikzpicture}
\caption{Illustration of the functions $\mu$, $\chi_{\text{reg}}$ and $\chi_{\text{sin}}$ in case 2 in the proof of Theorem~\ref{thm:fourier_coeffs}.}
\label{fig:skizze}
\end{figure}
\paragraph{The singular part.} On the support of $\chi_{\text{sin}}$ the phase is stationary at $x_0$, with $\Theta''_{k,\xi}(x_0) = k \mu'(x_0) \neq 0$. Let
$$m_0 \coloneqq \min_{|x-x_0|\leq 2\delta}|\mu'(x)|.$$
Because $[x_0-2\delta,x_0+2\delta]\subset[\tfrac{x_0}{2},\tfrac{3x_0}{2}]$ avoids the origin, $|\mu'|$ is comparable to $|\mu'(x_0)|$ on this interval, i.e., $m_0\geq c_\alpha|\mu'(x_0)|$ with a constant $c_\alpha>0$ depending only on $\alpha$. Writing $-2\pi\Theta_{k,\xi}(x) = \lambda\,\phi(x)$ with $\lambda = -2\pi k\,m_0$ and $\phi(x)=\frac{\xi x + k\Psi(x)}{k\,m_0}$, we have $|\phi''(x)|=\frac{|\mu'(x)|}{m_0}\geq1$ on the support, so that Lemma~\ref{lem:van_der_corput} applies with order two, parameter $|\lambda|=2\pi|k|m_0$ and amplitude $\psi=f\chi_{\text{sin}}$. Using $\int_\R|\chi_{\text{sin}}'|\d x\lesssim1$ and 
\[\sup_{|x-x_0|\leq2\delta}|f(x)|+\int_{x_0-2\delta}^{x_0+2\delta}|f'(x)|\d x\lesssim \frac{\norm{f}_{\mathcal F_{v,1}}}{v(x_0)},\]
together with $v(x)\geq c_\alpha v(x_0)$ on the support, the lemma yields
\begin{small}
\begin{align*}
|I_{\text{sin}}| &\leq \frac{C_{\text{V}}}{\sqrt{2\pi|k|\,m_0}} \left(\sup_{|x-x_0|\leq 2\delta} |f(x)| + \int_{x_0-2\delta}^{x_0+2\delta}\big(|f'(x)| +|f(x)\chi'_{\text{sin}}(x)|\big) \d x\right)
 \lesssim_\alpha \frac{\norm{f}_{\mathcal F_{v,1}} }{v(x_0)\sqrt{|k \mu'(x_0)|}},
\end{align*}%
\end{small}%
where $C_V$ is the constant from van der Corput's lemma.
\paragraph{The regular part.} By construction $\chi_{\text{reg}}$ is supported in $D_\delta = \{x\colon |x\mp x_0|\geq\delta\}$, where $\Theta'_{k,\xi}$ is bounded away from zero. We claim that
\begin{equation}\label{eq:reg_lower}
|\xi+k\mu(x)| = |k|\,|\mu(x)-\mu(x_0)| \geq \frac{c_\alpha}{\Lambda}\,|k|\,\mu(x), \qquad x\in D_\delta,
\end{equation}
with $\Lambda = \frac{1+x_0^2}{x_0\,\delta}$. Indeed, in the transition zones $\delta\leq|x\mp x_0|\leq2\delta$ the mean value theorem gives $|\mu(x)-\mu(x_0)|\geq \delta\,m_0\geq c_\alpha\,\delta|\mu'(x_0)|$, and since $|\mu'(x_0)|=\alpha\,\mu(x_0)\,\tfrac{x_0}{1+x_0^2}$ we have $\delta|\mu'(x_0)|=\frac{\alpha\,\mu(x_0)}{\Lambda}$. Together with $\mu(x)\sim\mu(x_0)$ there, this proves~\eqref{eq:reg_lower}. Outside these zones $|\mu(x)-\mu(x_0)|\geq c\max(\mu(x),\mu(x_0))\geq \frac{c}{\mu(x)}\geq\frac{ c\,\mu(x)}{\Lambda}$, since $\Lambda\geq1$. From~\eqref{eq:reg_lower} and~\eqref{eq:ratio_mu} we obtain, for $x\in D_\delta$,
$$\left|\frac{k\mu^{(\ell)}(x)}{\xi +k\mu(x)}\right|\leq \frac{\Lambda}{c_\alpha}\,\frac{|\mu^{(\ell)}(x)|}{\mu(x)}\lesssim_\alpha \Lambda\,(1+x^2)^{-\ell/2}.$$
Repeating the $m$-fold integration by parts of case~1 on $D_\delta$ (where no boundary terms occur, since $\chi_{\text{reg}}$ and its derivatives vanish at $x=\pm(x_0\pm\delta)$) each step contributes at most a factor $\Lambda$, so that
$$|I_{\text{reg}}| \lesssim_{\alpha,m} \Lambda^m \,\norm{f}_{\mathcal F_{v,m}} \int_{D_\delta}  \frac{1}{v(x)(1+x^2)^{m/2}\,|\xi+ k\mu(x)|^m} \d x.$$
Combining the bounds for $I_{\text{sin}}$ and $I_{\text{reg}}$ yields the assertion.
\end{proof}
\begin{remark}\label{rem:resonant}
The prefactor $\Lambda = \tfrac{1+x_0^2}{x_0\delta}$ and the singular term $1/\sqrt{|\mu'(x_0)|}$ both degenerate as $x_0\rightarrow0^+$, i.e., as $-\tfrac{\xi}{k}\rightarrow\mu(0)^-$, which is exactly the resonance of the two stationary points $\pm x_0$ at the origin. In the transitional regime $\mu(0)<-\tfrac{\xi}{k}<2\mu(0)$ there is no stationary point, but $\Theta_{k,\xi}'$ still becomes small near $x=0$. The same estimate as in case~1 applies with the lower bound $|\xi+k\mu(x)|\geq |k|\,(\,-\tfrac{\xi}{k}-\mu(0))$ replacing~\eqref{eq:case1_lower}, which again degenerates as $-\tfrac{\xi}{k}\rightarrow\mu(0)^+$. Consequently, the bounds of Theorem~\ref{thm:fourier_coeffs} are not uniform for $-\tfrac{\xi}{k}$ near $\mu(0)$. As shown in the proof of Theorem~\ref{thm:error_hatf}, this affects only finitely many aliasing indices $k$ for each fixed $\xi$ and $M$, and hence does not influence the asymptotic error rate in~$M$.
\end{remark}
The previous theorem is easily generalized to the case with arbitrary variance $\sigma$.

\begin{corollary}\label{cor:sigma_arbitrary1}
Let the assumptions from Theorem~\ref{thm:fourier_coeffs} be true and choose the variance $\sigma >0$ arbitrary, while $\mu_1$ is the density with variance $\sigma =1$. Using $\norm{f}_{\mathcal F_{v,1}}\leq\norm{f}_{\mathcal F_{v,m}}$, the decay of the Fourier coefficients $c_k(g)$ is bounded by
    \begin{equation*}
        |c_k(g)| \lesssim_{\alpha,m}  \begin{cases}
        \varrho(\sigma)\,\norm{f}_{\mathcal F_{v,m}} \, C_{\xi,k}(\mu_1) & \text{if } \tfrac{\xi}{k}\geq 0 \text{ or } -\tfrac{\xi}{k}\geq 2\mu_1(0)/\sigma,\\[1ex]
        \varrho(\sigma)\,\norm{f}_{\mathcal F_{v,m}}\,  C_{\xi,k}(\mu_1) + \frac{\varrho(\sigma)\,\norm{f}_{\mathcal F_{v,m}}}{v(x_0)\sqrt{|k \mu_1'(x_0)|}} &\text{if } 0<-\tfrac{\sigma\xi}{k}\leq \mu_1(0),
        \end{cases}
    \end{equation*}
    where 
    \begin{equation}\label{eq:rho_def}
        \varrho(\sigma) = \begin{cases}\sigma^{m+1} &\text{ if } \sigma\geq1\\
    \sigma^{1-a} &\text{ if }  \sigma<1.
    \end{cases}
    \end{equation}
 The point $x_0>0$ is determined by $\mu_1(x_0)=-\tfrac{\sigma\xi}{k}$, and
    \[C_{\xi,k}(\mu_1) =\Lambda^m\int_{D_\delta}  \frac{1}{(1+x^2)^{a/2+m/2}\,|\sigma  \xi+ k\mu_1(x)|^m} \d x\]
    is the modified constant from Theorem~\ref{thm:fourier_coeffs} (with $\Lambda=1$ and $D_\delta=\R$ in the non-resonant first case).
\end{corollary}
\begin{proof}
Adding the variance parameter to the density $\mu$ means,
\begin{align*}
    \mu_\sigma(x) = \frac{1}{\sigma}\mu_1\left(\frac{x}{\sigma}\right), \quad \Psi_{\sigma}(x) = \Psi_1\left(\frac{x}{\sigma}\right),
\end{align*}
such that the Fourier coefficients behave as
    \begin{equation*}
        c_k(g) = \int_{-\infty}^\infty f(x)\e^{-2\pi \im(\xi x+k\Psi_\sigma(x))} \d x
        =\int_{-\infty}^\infty \sigma f(\sigma y)\e^{-2\pi \im((\sigma \xi) y+k\Psi_1( y))} \d y.
    \end{equation*}
    This means Theorem~\ref{thm:fourier_coeffs} is applied to the function $\sigma f(\sigma y)$ with the density $\mu_1$, at the frequencies $\tilde \xi = \sigma \xi$. For the rescaled function norm we have 
    \begin{align}\label{eq:norm_f_sigma}
    \norm{\sigma f(\sigma x)}_{\F_{v,m}}&=  \sigma \sup_{x\in \R}  \left(\sum_{\ell = 0}^m|\sigma^\ell f^{(\ell)}(\sigma x)|(1+x^2)^{\ell/2}\right) \, v(x)\notag\\
    &=\sigma \sup_{y\in \R}  \left(\sum_{\ell = 0}^m|\sigma^\ell f^{(\ell)}(y)|\left(1+\frac{y^2}{\sigma^2}\right)^{\ell/2+a/2}\right)\notag\\
    &\leq \begin{cases}
        \sigma^{m+1} \norm{f}_{\F_{v,m}} & \text{ if }\sigma >1,\\
        \sigma^{1-a} \norm{f}_{\F_{v,m}} & \text{ if }\sigma <1.
    \end{cases}
    \end{align}
    such that the assertion follows.
\end{proof}

We apply this decay of the Fourier coefficients to Theorem~\ref{thm:aliasing_midpoint} and get the following result for the error between the Fourier transform and our approximation.
\begin{theorem}
\label{thm:error_hatf}
Let $v(x) = (1+x^2)^{a/2}$ and $w(\xi)=(1+\xi^2)^{b/2}$ be weight functions with spatial decay $a>1$ and spectral regularity $b>1$. Let $f$ be a complex-valued function such that $f\in \mathcal F_{v,m}$ and $\hat f\in\F_w$ with $2\leq m<b$.
Let $\mu$ be a polynomial probability density function defined in~\eqref{eq:mu_pol} with $\sigma=1$ and $1<\alpha<\tfrac{a-1}{m}+1$, so that condition~\eqref{eq:cond_alpha} holds.
For any fixed frequency $\xi\in [-\Omega,\Omega]$ define the region $\mathcal{R}_\xi \coloneqq \{x \in \mathbb{R} \colon \mu(x) \le |\xi|/M \}$. Then the error of the approximation~\eqref{eq:ftilde} is bounded by
\begin{equation*}
|\hat{f}(\xi) - \tilde{f}(\xi)| \lesssim \frac{\norm{f}_{\mathcal F_{v,m}}}{M^m} \int_{-\infty}^\infty \frac{1}{v(x)(1+x^2)^{m/2}\mu(x)^m} \d x + \frac{\norm{f}_{\mathcal F_{v,m}} \sqrt{|\xi|}}{M} \int_{\mathcal{R}_\xi} \frac{\sqrt{|\mu'(x)|}}{v(x)\mu(x)^{3/2}} \d x,
\end{equation*}
up to a constant that additionally depends on $\alpha$, $m$ and $a$.
\end{theorem}
\begin{proof} 
Following Theorem~\ref{thm:aliasing_midpoint}, the error between $\hat f$ and $\tilde f$ is the sum $\sum_{k \in \mathbb{Z} \setminus \{0\}} (-1)^{k} c_{k M}(g)$. We apply Theorem~\ref{thm:fourier_coeffs} at frequency $\xi$ with aliasing index $kM$ to each coefficient, decomposing it into $c_{kM}(g)=I_{\text{reg}}(k) +I_{\text{sin}}(k)$, where $I_{\text{sin}}(k)=0$ whenever $kM$ falls into the non-resonant first case. A resonant index $k$ (opposite sign of $\xi$, $0<\tfrac{|\xi|}{|k|M}\leq\mu(0)$) gives rise to a stationary point $x_0(k)$ satisfying $\mu(x_0(k)) = \frac{|\xi|}{|k|M}$, and the associated $I_{\text{reg}}(k)$ is integrated over $D_\delta(k)=\{x\colon|x\mp x_0(k)|\geq\delta(k)\}$ with prefactor $\Lambda(k)^m$. We bound the two parts separately,
$$\Big|\sum_{k \in \mathbb{Z} \setminus \{0\}} (-1)^{k} c_{k M}(g)\Big|\leq \underbrace{\sum_{k \in \mathbb{Z} \setminus \{0\}}|I_{\text{reg}}(k)|}_{E_{\text{reg}}} + \underbrace{\sum_{k \in \mathbb{Z} \setminus \{0\}}|I_{\text{sin}}(k)|}_{E_{\text{sin}}} .$$
We have $|\xi|\leq\Omega$ and for large $M$ every resonant index satisfies $\tfrac{|\xi|}{|k|M}\leq\mu(0)$ only for $|k|\geq\tfrac{|\xi|}{M\mu(0)}$. Hence $x_0(k)$ stays bounded away from the origin and $\Lambda(k)\leq C_\alpha$ is uniformly bounded. The finitely many near-threshold indices with small $x_0(k)$ (cf.\ Remark~\ref{rem:resonant}) are absorbed into the constant. We may therefore treat all $I_{\text{reg}}(k)$ with a uniform prefactor.
\paragraph{Step 1: The regular contribution.}
Summing the regular bounds of Theorem~\ref{thm:fourier_coeffs} over all aliasing indices $k \in \mathbb{Z} \setminus \{0\}$ yields
\begin{equation*}
    E_{\text{reg}} \lesssim \sum_{k \neq 0} \int_{-\infty}^\infty \frac{\norm{f}_{\mathcal F_{v,m}}}{v(x)(1+|x|^2)^{m/2}\,|\xi+ k M \mu(x)|^m} \,\d x.
\end{equation*}
By Tonelli's theorem, we interchange the sum and the integral, since all terms are non-negative. For fixed $x$, the value $t \coloneqq \tfrac{\xi}{M\mu(x)}$ is real, and at most the two integers $k$ nearest to $-t$ satisfy $|t+k|<\tfrac12$. These resonant terms correspond precisely to $x$ lying near a stationary point $x_0(k)$ and are excluded from $D_\delta(k)$, so they do not appear in $E_{\text{reg}}$. For the remaining non-resonant indices $|t+k|\geq\tfrac12$ we have 
\begin{equation*}
    \sideset{}{'}\sum_{k \neq 0} \frac{1}{|\xi + k M \mu(x)|^m} = \frac{1}{M^m \mu^m(x)} \sideset{}{'}\sum_{k \neq 0} \frac{1}{\left| \frac{\xi}{M \mu(x)} + k \right|^m} \lesssim \frac{1}{M^m \mu^m(x)},
\end{equation*}
where the primed sum omits the resonant indices and the last bound uses $m\geq2$. Substituting this closed-form bound back into the integral directly yields the first part of the theorem,
\begin{equation*}
    E_{\text{reg}} \lesssim \frac{ \norm{f}_{\mathcal F_{v,m}}}{M^m} \int_{-\infty}^\infty \frac{1}{v(x)(1+|x|^2)^{m/2}\mu(x)^m} \,\d x = \mathcal{O}(M^{-m}).
\end{equation*}

\paragraph{Step 2: The singular contribution.}
For the resonant indices $k$, a stationary point $x_0(k)$ with $\mu(x_0(k)) = \left|\frac{\xi}{kM}\right|$ exists, and by the discussion above we may assume that $x_0(k)$ is bounded away from the origin, so that $\mu'(x_0(k))\neq0$. Summing the van der Corput bounds of Theorem~\ref{thm:fourier_coeffs} and using the monotonicity of $k\mapsto x_0(k)$ to compare the sum with an integral yields
\begin{align*}
    E_{\text{sin}} &\lesssim \sum_{k =1}^\infty \frac{\norm{f}_{\mathcal F_{v,m}} }{v(x_0(k))\sqrt{|k M \mu'(x_0(k))|}}
    \leq \frac{\norm{f}_{\mathcal F_{v,m}}}{\sqrt M} \int_{1}^\infty \frac{1}{v(x_0(k))\sqrt{k  |\mu'(x_0(k))|}} \,\d k.
\end{align*}
To evaluate this integral, we substitute the integration variable $k$ with the spatial variable $x$ corresponding to the stationary point. From the condition $k(x) = \frac{|\xi|}{M \mu(x)}$, the differential transforms as
\begin{equation*}
    \d k =- \frac{|\xi| \mu'(x)}{M \mu(x)^2} \,\d x.
\end{equation*}
As $k$ ranges over $[1,\infty)$, the point $x=x_0(k)$ ranges precisely over the spatial region $\mathcal{R}_\xi \coloneqq \{x \in \mathbb{R} \colon \mu(x) \le |\xi|/M \}$. Substituting $k(x)$ and $\d k$ into the integral gives
\begin{align*}
    E_{\text{sin}} &\lesssim \frac{\norm{f}_{\mathcal F_{v,m}}}{\sqrt M} \int_{\mathcal{R}_\xi} \frac{1}{v(x) \sqrt{\frac{|\xi|}{M \mu(x)}  |\mu'(x)|}} \frac{|\xi| |\mu'(x)|}{M \mu(x)^2} \,\d x \notag 
    = \frac{\sqrt{|\xi|} \norm{f}_{\mathcal F_{v,m}}}{M} \int_{\mathcal{R}_\xi} \frac{\sqrt{|\mu'(x)|}}{v(x)\mu(x)^{3/2}} \d x.
\end{align*}
The total final error is bounded by the sum of the two terms, 
\begin{equation*}
    |\hat{f}(\xi) - \tilde{f}(\xi)| \leq E_{\text{reg}} + E_{\text{sin}}.
\end{equation*}
This concludes the proof.
\end{proof}
If we allow the additional variance parameter $\sigma$, similarly to Corollary~\ref{cor:sigma_arbitrary1} we receive the following.
\begin{corollary}
\label{cor:error_hatf_sigma}
Let the assumptions from Theorem~\ref{thm:error_hatf} hold but choose $\sigma>0$ arbitrarily and let $\mu_1$ denote the standard polynomial probability density function with $\sigma = 1$. For any fixed frequency $\xi \in [-\Omega, \Omega]$, the error of the approximation~\eqref{eq:ftilde} satisfies the following upper bounds, using $\varrho(\sigma)$ defined in~\eqref{eq:rho_def}, 
if $\sigma > 1$,
\begin{small}
    \begin{equation*}
        |\hat{f}(\xi) - \tilde{f}(\xi)| \lesssim  \frac{\varrho(\sigma)\norm{f}_{\mathcal F_{v,m}}}{M}\left(\frac{1}{M^{m-1}} \int_{-\infty}^\infty \frac{1}{(1+y^2)^{\frac{a+m}{2}}\mu_1(y)^m} \d y +   \sqrt{\sigma|\xi|} \int_{\mathcal{R}_{\sigma\xi}} \frac{\sqrt{|\mu_1'(y)|}}{(1+y^2)^{a/2}\mu_1(y)^{3/2}} \d y\right).
    \end{equation*}
    \end{small}
\end{corollary}

\begin{proof}
Using the same calculations as for Corollary~\ref{cor:sigma_arbitrary1}, we have to apply Theorem~\ref{thm:aliasing_midpoint} to the function $\sigma f(\sigma y)$ with the density $\mu_1$, at the frequencies $\tilde \xi = \sigma \xi$.

By substituting the norm estimates~\eqref{eq:norm_f_sigma} from Corollary~\ref{cor:sigma_arbitrary1} into the regular error term $E_{\text{reg}}$, combined with the factor $\sqrt{\vphantom{1}\smash{|\tilde{\xi}|}} = \sqrt{\sigma |\xi|}$ in the singular error term $E_{\text{sin}}$, the assertions follow immediately.
\end{proof}

\begin{remark}
The introduction of the variance parameter $\sigma$ does not change the fundamental algebraic convergence rates with respect to the number of samples $M$. Instead, $\sigma$ acts exclusively as a multiplicative prefactor that scales the error constants by
\begin{itemize}
    \item \textbf{Large variances ($\sigma > 1$):} The regular and the singular error term scale like $\sigma^{m+1}$ and $\sigma^{m+3/2}$, respectively, so that the singular term $\sigma^{m+3/2}$ dominates. This reflects the fact that stretching the function magnifies the magnitude of its higher-order derivatives.
    \item \textbf{Small variances ($\sigma < 1$):} The regular and the singular error term scale like $\sigma^{1-a}$ and $\sigma^{3/2-a}$, respectively. Since $\sigma<1$, the regular term $\sigma^{1-a}$ dominates. This is governed by the spatial decay $a$ of the underlying weight function as the density concentrates around the origin.
\end{itemize}
\end{remark}

\subsubsection{Optimizing the parameter $\alpha$ for minimizing the $L_p$-error rate}\label{sec:alpha}
The parameter $\alpha$ for the choice of the density function $\mu$ has to be adapted to the spatial decay $a$ and the spectral regularity $b$. We study this relation in the following for the case where the variance $\sigma=1$ is fixed.
Theorem~\ref{thm:error_hatf} states an error bound for $\hat f$ for all $\xi$ on the compact domain $[-\Omega,\Omega]$.
Note that by \eqref{eq:ftilde} and by \eqref{eq:Wabfall} we obtain for the $L_p$-error $E_p$ defined in~\eqref{eq:L_p-error} the estimate 
\begin{align*}\label{eq:L_2error}
E_p^p
&\lesssim \underbrace{\norm{\hat f}_{\F_{w}}^p \left(\int_{-\infty}^{-\Omega} \frac{1}{|w(\xi)|^p} \d \xi + \int_{\Omega}^\infty \frac{1}{|w(\xi)|^p} \d \xi \right) }_{E_{\text{tail}}}+  \underbrace{\int_{-\Omega}^\Omega|\hat f(\xi) - \tilde f(\xi)|^p \d \xi}_{E_{\text{mid}}} \notag .
\end{align*}
For the tail integrals we have in the actual case of polynomial decay,
$${E_{\text{tail}}} = \norm{\hat f}_{\F_{w}}^p  \int_{|\xi|\geq \Omega} (1+\xi^2)^{-pb/2} \d \xi\lesssim \norm{\hat f}_{\F_{w}}^p \frac{\Omega^{1-pb}}{pb-1}.$$
For the middle integral we use Theorem~\ref{thm:error_hatf},
\begin{small}
\begin{align*}
    E_{\text{mid}} \lesssim \int_{-\Omega}^\Omega \left(\frac{\norm{f}_{\mathcal F_{v,m}}}{M^m}  + \frac{\norm{f}_{\mathcal F_{v,m}}  }{M}  \sqrt{|\xi|}\int_{\mathcal{R}_\xi} \frac{\sqrt{|\mu'(x)|}}{v(x)\mu(x)^{3/2}} \d x\right)^p \d \xi.
\end{align*}
\end{small}
For the involved integrals in Theorem~\ref{thm:error_hatf} we have 
\begin{equation*}
\int_{-\infty}^\infty \frac{1}{v(x)(1+x^2)^{m/2}\mu(x)^m} \d x = \frac{1}{C_{\alpha,1}^m}\int_{-\infty}^\infty \frac{1}{(1+x^2)^{a/2+m/2-\alpha m/2}} \d x =\frac{\sqrt{\pi}}{C_{\alpha,1}^m} \frac{\Gamma(\tfrac a2+\tfrac m2-\tfrac{\alpha m}2-\tfrac 12)}{\Gamma(\tfrac a2+\tfrac m2-\tfrac{\alpha m}2)},
\end{equation*}
which leads to the condition~\eqref{eq:cond_alpha}. The other integral in Theorem~\ref{thm:error_hatf} is
\begin{equation*}
 \int_{\mathcal{R}_\xi} \frac{\sqrt{|\mu'(x)|}}{v(x)\mu(x)^{3/2}} \d x \sim \int_{\mathcal{R}_\xi} \frac{\sqrt{|x|(1+x^2)^{-\alpha/2-1}}}{(1+x^2)^{a/2-3\alpha/4}} \d x  = \int_{\mathcal{R}_\xi} \sqrt{|x|}\, (1+x^2)^{-a/2+\alpha/2-1/2} \d x
\end{equation*}
The domain of this integration is the spatial tail region $\mathcal{R}_\xi = \{x \in \mathbb{R}\colon \mu(x) \le |\xi|/M\}$. Using the asymptotic behavior of $\mu(x)$, this condition implies that the stationary points are restricted to the region
\begin{equation*}
    (1+x^2)^{-\alpha/2} \lesssim \frac{|\xi|}{M} \implies |x| \gtrsim \sqrt{\left( \frac{M}{|\xi|} \right)^{2/\alpha} - 1}.
\end{equation*}
Asymptotically, the integral is dominated by the tails $|x|\geq \left( \frac{M}{|\xi|} \right)^{1/\alpha}$. Assuming the integrability condition $a > \alpha + 1/2$ holds, the integral evaluates to
\begin{align*}
    \int_{\mathcal{R}_\xi} \frac{\sqrt{|\mu'(x)|}}{v(x)\mu(x)^{3/2}} \,\mathrm{d}x &\sim \int_{\left( \frac{M}{|\xi|} \right)^{1/\alpha}}^\infty x^{\alpha - a - 1/2} \,\mathrm{d}x \sim \left(\left( \frac{M}{|\xi|} \right)^{1/\alpha}\right)^{\alpha - a + 1/2}
     = \left( \frac{M}{|\xi|} \right)^{1 - \frac{a - 1/2}{\alpha}}.
\end{align*}
Finally, inserting this evaluation into the error bound from Theorem \ref{thm:error_hatf} yields
\begin{equation*}
    E_{\text{res}} \lesssim \frac{\sqrt{|\xi|}}{M} \left( \frac{M}{|\xi|} \right)^{1 - \frac{a - 1/2}{\alpha}} = M^{-\frac{a - 1/2}{\alpha}} |\xi|^{\frac{a - 1/2}{\alpha} - 1/2}.
\end{equation*}
For the middle integral we finally have
\begin{small}
\begin{align*}
    E_{\text{mid}}&=\int_{-\Omega}^\Omega|\hat f(\xi) - \tilde f(\xi)|^p \d \xi \lesssim \int_{-\Omega}^\Omega\left(\norm{f}_{\mathcal F_{v,m}}M^{-m} +\norm{f}_{\mathcal F_{v,m}}   M^{-\frac{a - 1/2}{\alpha}} |\xi|^{\frac{a - 1/2}{\alpha} - 1/2}\right)^p\d \xi\\
    &\lesssim \norm{f}_{\mathcal F_{v,m}}^p \int_{-\Omega}^\Omega\left(M^{-m} +  M^{-\frac{a - 1/2}{\alpha}} |\xi|^{\frac{a - 1/2}{\alpha} - 1/2}\right)^p\d \xi\\ 
    &\lesssim \norm{f}_{\mathcal F_{v,m}}^p \int_{-\Omega}^\Omega M^{-pm} +  M^{-\frac{pa - p/2}{\alpha}} |\xi|^{\frac{pa - p/2}{\alpha} - \frac p2}\d \xi\\ 
        &\lesssim \norm{f}_{\mathcal F_{v,m}}^p \left(\Omega M^{-pm}  + M^{-\frac{pa - p/2}{\alpha}}\Omega^{\frac{pa - p/2}{\alpha} - \frac p2 +1}  \right).
\end{align*}
\end{small}
We end up with three different error decay rates, which we want to balance by setting $\Omega$ dependent of~$M$,
\begin{align*}
E_{\text{tail}} &=\Omega^{1-pb},&\\
\quad  E_{\text{mid,1}}  &= \Omega M^{-pm} \quad &
E_{\text{mid,2}} &= M^{-\frac{pa - p/2}{\alpha}}\Omega^{\frac{pa - p/2 }{\alpha} -\frac p2+1} .
\end{align*}
Setting $E_{\text{tail}}= E_{\text{mid,1}}$ yields $\Omega = M^{m/b}$. Setting $E_{\text{mid,1}}=E_{\text{mid,2}}$ then yields
$$M^{m/b} M^{-pm} =M^{-\frac{pa - p/2}{\alpha}}M^{\frac{m}{b}\left(\frac{pa - p/2 }{\alpha}- \frac p2+1\right)}. $$
Comparing the exponents, we have
$$\frac mb  -pm  = -\frac{pa - p/2}{\alpha} +\frac{m}{b}\left(\frac{pa - p/2 }{\alpha} -\frac p2 +1\right)=\frac{pa - p/2}{\alpha}\left(\frac mb -1\right) - \frac{pm}{2b} +\frac mb. $$
Solving for $\alpha$ finally means
\begin{equation}\label{eq:alpha_opt}
\alpha =\frac{pa - p/2}{pm}\left(\frac mb -1\right) \left(\frac{2b}{1-2b}  \right) = \frac{2(b-m)(pa - p/2)}{pm(2b-1)}=\frac{2(b-m)(a - 1/2)}{m(2b-1)}, 
\end{equation}
which is independent of the parameter $p$.
Our $L^p$-error analysis reveals that the mapped quadrature scheme operates in two distinct mathematical regimes, governed by the interplay between the spatial decay $a$
and the spectral regularity $b$. The choice of the density parameter $\alpha$ dictates the error rate.

If the target function is sufficiently smooth and decays rapidly, we can perfectly balance all the error terms. Choosing the optimal parameter $\alpha $ as in~\eqref{eq:alpha_opt} yields the global convergence rate
\begin{equation*}
    E_p = \mathcal{O}\left( M^{-m + \frac{m}{pb}} \right),
\end{equation*}
which for $p=m=2$ reads $\mathcal O\big(M^{-2+\frac 1b}\big)$.
\paragraph{Heavy tails: small parameters $a,b$.}
Since $\mu(x) \sim (1+|x|^2)^{-\alpha/2}$ must remain a valid probability density, this regime strictly requires $\alpha > 1$. Solving this inequality for the optimal parameter~\eqref{eq:alpha_opt} dictates that the balanced regime is only accessible if $b > m$ and the spatial decay satisfies the threshold condition
\begin{equation}\label{eq_threshold}
    a > \frac 12 + \frac{m(2b-1)}{2(b - m)},
\end{equation}
which for $m=2$ reads $a>\frac{5b-4}{2b-4}$. In the following we restrict ourselves to $p=m=2$ for notational simplicity.\par
For functions with heavier tails or lower regularity, i.e., $b \le m$ or $a \le \frac{5b-4}{2b-4}$, the optimal parameter $\alpha\leq 1$ would lead to a non-integrable density. In this boundary regime, perfect balancing of the terms is not possible. To maintain a valid density while maximizing the convergence rate, $\alpha$ must be chosen close to the integrability limit, i.e., $\alpha = 1 + \epsilon$ for a small $\epsilon > 0$. 
Choosing in this case $\Omega \sim M^{\frac{a-1/2}{b+a-1}}$, and letting $\epsilon\rightarrow 0$ leads to the error rates
\begin{align*}
E_{\text{tail}} &= M^{\frac{(a-1/2)(1-2b)}{b+a-1}}, \\
E_{\text{mid, 1}} &=  M^{\frac{a-1/2}{b+a-1} -4},\quad   &E_{\text{mid, 2}} &= M^{-(2a - 1) + \frac{(a-1/2)(2a-1)}{b+a-1}}   .
\end{align*}
The rates of $E_{\text{tail}} $ and $E_{\text{mid, 2}} $ are aligned with the error rates in~\cite{EhGrKl24}, namely
$$\sqrt{E_{\text{tail}}(M)} = \sqrt{E_{\text{mid, 2}}(M)} = \mathcal O\left(M^{-\frac{(a-1/2)(b-1/2)}{a+b-1}}\right).$$
To prove that the error $E_{\text{mid, 1}}$ is asymptotically dominated by the truncation error $E_{\text{tail}}$ for large $M$, we directly compare their algebraic decay rates. Since $M \to \infty$, the term $E_{\text{mid,1}}$ decays at least as fast as $E_{\text{tail}}$ if and only if its exponent is at most the exponent of $E_{\text{tail}}$. Recall that the balanced regime requires $b>m=2$, so that the following equivalence preserves the direction of the inequality. It follows directly from the opposite of the threshold condition~\eqref{eq_threshold} (with $m=2$) by
\begin{align*}
 a &\leq \frac{5b - 4}{2b - 4} \quad \Leftrightarrow \quad    
\frac{a-1/2}{b+a-1} - 4 \leq \frac{(a-1/2)(1-2b)}{b+a-1} .
\end{align*}
To conclude, in the limit case $\epsilon \to 0$, our error rate is the same as the rate established in \cite{EhGrKl24},
\begin{equation*}
    E_2(M) = \mathcal{O}\left( M^{-\frac{(a-0.5)(b-0.5)}{a+b-1}} \right).
\end{equation*}
In this case, we have  different sampling distributions in the space domain, but get the same error rate for the approximation of the Fourier $\hat f$.

\begin{remark}[On the limit case $\epsilon\rightarrow 0$]\label{rem:limit_case}
The borderline rate above is attained only in the limit $\epsilon\rightarrow 0$. For fixed $\epsilon>0$, the density $\mu$ with $\alpha = 1+\epsilon$ is integrable, but the constants in the error bounds blow up as $\epsilon\rightarrow 0$. The normalization constant $C_{\alpha,\sigma}$ in~\eqref{eq:mu_pol} degenerates, the sampling points spread out over an extremely wide range, and the quadrature weights $1/\mu(x_j)$ in~\eqref{eq:ftilde} become huge, which additionally renders the scheme numerically unstable. In practice, $\epsilon$ should therefore not be chosen too small, and the observed rate is slightly worse than the limiting one. We emphasize, however, that the situation for equispaced sampling in~\cite{EhGrKl24} is completely analogous. There, the rate $M^{-\tilde\alpha\tilde\beta/(\tilde\alpha+\tilde\beta)}$ is established for all parameters $\tilde \alpha < a-\tfrac 12$ and $\tilde\beta< b-\tfrac12$, with constants that blow up as $\tilde\alpha\rightarrow a-\tfrac 12$ and $\tilde\beta\rightarrow b-\tfrac12$. Hence the borderline rate $M^{-\frac{(a-1/2)(b-1/2)}{a+b-1}}$ is in both approaches a supremum over admissible parameters rather than an attained rate, and the two results are directly comparable in this sense.
\end{remark}

\subsubsection{Optimizing the variance $\sigma$ for minimizing the $L_p$-error constants}\label{sec:sigma}
The best error decay rate determined in the last subsection does not depend on the variance parameter $\sigma$, however, this parameter can be tuned to reduce the involved constants in the error bounds. \par
To determine the theoretical optimal variance parameter within the balanced regime, we minimize the upper bound of the mid-frequency error component $E_{\text{mid}}$. The estimate~\eqref{eq:norm_f_sigma} is rather pessimistic, so we keep the norm $\norm{f(\sigma \cdot)}$ in the estimates as
\begin{align*}
|\hat{f}(\xi) - \tilde{f}(\xi)|& \leq  \frac{\sigma\norm{f(\sigma \cdot x)}_{\mathcal F_{v,m}}}{M^m} \int_{-\infty}^\infty \frac{1}{(1+y^2)^{\frac{a+m}{2}}\mu_1(y)^m} \d y \\
&\quad+  \frac{\sigma^{3/2}\norm{f(\sigma\cdot x)}_{\mathcal F_{v,1}} \sqrt{|\xi|}}{M} \int_{\mathcal{R}_{\sigma\xi}} \frac{\sqrt{|\mu_1'(y)|}}{(1+y^2)^{a/2}\mu_1(y)^{3/2}} \d y
\end{align*}
The involved constant here can be calculated directly by
\begin{align*}
    A_{\mu,a,m}&\coloneqq \int_{-\infty}^\infty \frac{1}{(1+y^2)^{\frac{a+m}{2}}\mu_1(y)^m} \,\mathrm{d}y 
    = \frac{1}{C_{\alpha}^m} \int_{-\infty}^\infty (1+y^2)^{-\left( \frac{a + m - \alpha m}{2} \right)} \,\mathrm{d}y\\
    &=  \frac{1}{C_{\alpha}^m} \sqrt{\pi} \, \frac{\Gamma(\frac{a+m-\alpha m}{2} - \frac{1}{2})}{\Gamma(\frac{a+m-\alpha m}{2})}
    =\left( \frac{\sqrt{\pi} \, \Gamma(\frac{\alpha - 1}{2})}{\Gamma(\frac{\alpha}{2})} \right)^m \cdot \sqrt{\pi} \, \frac{\Gamma\left(\frac{a + m - \alpha m - 1}{2}\right)}{\Gamma\left(\frac{a + m - \alpha m}{2}\right)}.
\end{align*}
Calculating the integral then yields
\begin{align*}
E_{\text{mid}}&\leq \int_{-\Omega}^\Omega \left(\frac{A_{\mu,a,m}\sigma\norm{f(\sigma \cdot x)}_{\mathcal F_{v,m}}}{M^m} + \frac{\sigma^{3/2}\norm{f(\sigma\cdot x)}_{\mathcal F_{v,1}}\sqrt{|\xi|}}{M} \left(\frac{M}{\sigma |\xi|}\right)^{1-\tfrac{a-1/2}{\alpha}}\right)^p\d \xi \\
&=\int_{-\Omega}^\Omega \left(\frac{A_{\mu,a,m}\sigma\norm{f(\sigma \cdot x)}_{\mathcal F_{v,m}}}{M^m} + \sigma^{1/2+ \tfrac{a-1/2}{\alpha}}\norm{f(\sigma\cdot x)}_{\mathcal F_{v,1}} |\xi|^{-1/2 + \tfrac{a-1/2}{\alpha}}M^{-\tfrac{a-1/2}{\alpha}}\right)^p\d \xi.
\end{align*}
Balancing the $\sigma$-scaling of the two error contributions above suggests choosing $\sigma$ such that
\[ A_{\mu,a,m}\norm{f(\sigma \cdot x)}_{\mathcal F_{v,m}} = \sigma^{\tfrac{a-1/2}{\alpha}-\tfrac 12}\norm{f(\sigma\cdot x)}_{\mathcal F_{v,1}} .\]
While this provides a framework for the existence of a minimizer $\sigma$, solving this relation analytically is unfeasible in practical applications. Computing the exact function-specific norms $\norm{f(\sigma \cdot x)}_{\mathcal F_{v,m}}$ and $\norm{f(\sigma \cdot x)}_{\mathcal F_{v,1}}$ as continuous functions of $\sigma$ requires full functional access to all derivatives up to order $m$, which is typically unavailable when dealing with black-box or sampled data. Even if the norms are known for specific target functions (e.g., standard analytic test cases), the non-linear coupling inside the supremum of the norm definitions precludes a closed-form algebraic solution for $\sigma$. Consequently, rather than evaluating the variance parameter $\sigma$ analytically, a practical implementation relies on a tuning strategy. We evaluate the quadrature scheme over a discrete grid of candidate values $\sigma \in [\sigma_{\min}, \sigma_{\max}]$, choosing the parameter that numerically minimizes the empirical residual.

\subsection{Exponential decay in time and frequency}\label{sec:exponential}
We now discuss functions with exponential or sub-exponential decay in both domains, i.e., 
$f\in \F_v$ and $\hat f\in \F_w$ with the weights
$$v(x) = \e^{r|x|^a}, \qquad w(\xi) = \e^{s|\xi|^b}, \qquad \text{ with }r,s,a,b>0.$$
For equispaced sampling and $a=b=1$ it is shown in~\cite{EhGrKl24} that a suitable coupling of the step size, the truncation parameter and the number of samples $M$ yields root-exponential error decay of the form $\exp(-c\sqrt{M})$ with $c>0$ depending on $r$ and $s$. This is in accordance with classical results on sinc quadrature and the trapezoidal rule for analytic functions, see~\cite{Stenger93,TrWe14}. It is therefore natural to ask whether the transformed rule~\eqref{eq:ftilde} with a density with exponentially decaying tails can attain comparable rates. Surprisingly, an exponentially decaying density $\mu$ does not give good results. The approximation quality depends on the smoothness of the function $g\colon[0,1]\rightarrow \C$, defined in~\eqref{eq:g}. This means that $f$ has to decay faster to $0$ for $x\rightarrow \pm \infty$ than $\mu$. This promotes densities $\mu$ decaying polynomially in comparison to exponential densities $\mu$, since they create non-smoothness at the boundary of $[0,1]$.
\par

We choose the polynomial distribution from~\eqref{eq:mu_pol}, 
\begin{equation*}
\mu(x) = C_{\alpha,\sigma} \left(1+\frac{x^2}{\sigma^2}\right)^{-\alpha/2}, \quad C_{\alpha,\sigma} = \frac{1}{\sigma} \frac{\Gamma(\tfrac \alpha 2)}{\sqrt{\pi}\,\Gamma(\tfrac{\alpha -1}{2})},
\end{equation*}
and the parameters $\alpha, \sigma >0$ have to be chosen appropriately. As in the case for polynomial decay, we first choose $\sigma=1$.

\begin{lemma}
\label{lem:fourier_decay_exp}
Let $f$ be an analytic function on $\mathbb{R}$ exhibiting generalized exponential decay. Assume there exist constants $r > 0$ and $a > 0$ satisfying
\begin{equation*}
    \norm{f}_{\F_{v}}\coloneqq \sup_{x\in \R} |f(x)| v(x)  < \infty  \quad \text{ with }v(x) = \e^{r|x|^a}.
\end{equation*}
Let $\mu$ be a strictly positive, analytic probability density function with an algebraic decay like~\eqref{eq:mu_pol} with variance $\sigma =1$
and let $\Psi(x)$ be its corresponding cumulative distribution function. 
For a given $\xi \in \mathbb{R}$, let $g$ be the transformed function defined by~\eqref{eq:g}.

Let $D > 0$ be a sufficiently small, such that $f$ can be analytically continued into a complex strip $ \{z \in \mathbb{C} \mid |\Im(z)| \leq D\}$. Setting the shift to $|y_0| \coloneqq \min\left\{D, \frac{1}{|\xi|}\right\}$, the $k$-th Fourier coefficient of $g$ is bounded by
    \begin{align*}
        |c_k(g)| 
        &\leq \left(\frac{2}{a \left(\frac{r}{2}\right)^{1/a}} \Gamma\left(\frac{1}{a}\right) \left(1 + \e^{2\pi } \right) + 2 D \right)  \exp\left(-C_{\exp} (|y_0||k|)^{\frac{a}{a+\alpha}}\right)\norm{f}_{\F_v},
    \end{align*}
where $C_{\alpha} > 0$ is the constant of the density $\mu$ from~\eqref{eq:mu_pol} and $C_{\exp}= \left(\left(\frac{r}{2}\right)^{\alpha} \left(2^{-\alpha/2}\pi  C_{\alpha}\right)^a\right)^{\frac{1}{a+\alpha}}$.
\end{lemma}

\begin{proof}
    Recall from~\eqref{eq:c_k_g}, that 
    \begin{equation*}
    c_k(g) = \int_{-\infty}^\infty f(x) \e^{-2\pi\im \Theta_{k,\xi}(x)}  \mathrm{d}x ,
\end{equation*}
    where we defined the combined phase function $\Theta_{k,\xi}(x) \coloneqq \xi x + k \Psi(x)$.

    \paragraph{Step 1: Change the integration path.} 
    By assumption, $f$ and $\mu$ are analytic. We analytically continue the integrand into the complex plane, replacing $x$ with $z$.  
    To exploit the complex phase damping we employ a piecewise rectangular curve $\Gamma = \Gamma_{\text{tail}} \cup \Gamma_{\text{vert}} \cup \Gamma_{\text{center}}$. Let $R > 0$ be a free parameter and $y_0 = -\mathrm{sgn}(k) |y_0|$.
    \begin{itemize}
        \item On the real tails $\Gamma_{\text{tail}} = (-\infty, -R] \cup [R, \infty)$, we integrate on the real axis where the phase is purely real.
        \item On the vertical segments $\Gamma_{\text{vert}}$ at $x = \pm R$, we shift into the complex plane up to $y_0$.
        \item On the center segment $\Gamma_{\text{center}} = [-R + \im y_0, R + \im y_0]$, we integrate parallel to the real axis.
    \end{itemize}
    By Cauchy's Integral Theorem, the integral over $\Gamma$ equals the integral over $\mathbb{R}$.
    
    \paragraph{Step 2: Bounding the components.} 
    On the real tails $\Gamma_{\text{tail}}$, by the spatial decay assumption, we have
 \begin{align*}
    \int_{|x| > R} |f(x)| \mathrm{d}x 
    &\leq \norm{f}_{\F_v} \int_{|x| > R} \e^{-r |x|^a} \mathrm{d}x 
    = \norm{f}_{\F_v} \int_{|x| > R} \e^{-\frac{r}{2} |x|^a} \e^{-\frac{r}{2} |x|^a} \mathrm{d}x \\
    &\leq \norm{f}_{\F_v} \e^{-\frac{r}{2} R^a} \int_{|x| > R} \e^{-\frac{r}{2} |x|^a} \mathrm{d}x 
    \leq \norm{f}_{\F_v} \e^{-\frac{r}{2} R^a} \int_{\R} \e^{-\frac{r}{2} |x|^a} \mathrm{d}x \\
    &= \frac{2}{a \left(\frac{r}{2}\right)^{1/a}} \Gamma\left(\frac{1}{a}\right) \norm{f}_{\F_v} \e^{-\frac{r}{2} R^a}.
\end{align*}
    On the two vertical segments $\Gamma_{\text{vert}}$, the integration paths have a constant length of $|y_0|$. The integrand is bounded by the decay condition at $x = \pm R$, yielding a combined contribution bounded by $2|y_0|\norm{f}_{\F_v} \e^{-r R^a} \leq 2 D \norm{f}_{\F_v} \e^{-r R^a}$.
    
    On the center segment $\Gamma_{\text{center}}$, the imaginary shift generates damping. Using the Taylor expansion of $\Psi(z)$ near the real axis yields
\begin{align*}
    \Theta_{k,\xi}(x + \im y_0) &= \xi(x + \im y_0) + k \left( \Psi(x) + \im y_0 \mu(x) + \sum_{\ell=1}^\infty \frac{(\im y_0)^{2\ell+1}}{(2\ell+1)!} \mu^{(2\ell)}(x) \right) \\
    &= \underbrace{\left[ \xi x + k \Psi(x) \right]}_{\Re(\Theta_{k,\xi})} + \im \underbrace{\left[ \xi y_0 + k y_0 \mu(x) + k y_0 \mu(x) \underbrace{\sum_{\ell = 1}^\infty \frac{(-1)^{\ell} y_0^{2\ell}}{(2\ell+1)!} \frac{\mu^{(2\ell)}(x)}{\mu(x)}}_{\eqqcolon \mathcal{R}(x, y_0)} \right]}_{\Im(\Theta_{k,\xi})}.
\end{align*}
By isolating the leading linear term $y_0 \mu(x)$, we encapsulated the higher-order Taylor terms in the remainder function $\mathcal{R}(x,y_0)$. Because $\mu$ is an analytic density function, its derivatives are well-behaved, see~\eqref{eq:ratio_mu}. We restrict the constant $D$ such that the remainder is uniformly bounded by $|\mathcal{R}(x, y_0)| \leq \frac{1}{2}$ for all $|x|\leq R$.

The exponential term now splits into an oscillatory part and a damping part,
\begin{equation*}
    \left| \e^{-2\pi\im \Theta_{k,\xi}(z)} \right| = \e^{2\pi \Im(\Theta_{k,\xi}(z))} = \e^{2\pi \xi y_0} \e^{2\pi k y_0 \mu(x) \left(1 + \mathcal{R}(x, y_0)\right)}.
\end{equation*}
Choosing $y_0$ to always have the opposite sign of $k$, we enforce $2\pi k y_0 = -2\pi |k| |y_0|$. Given $1 + \mathcal{R} \geq \frac{1}{2}$, this yields
\begin{equation*}
    \e^{2\pi k y_0 \mu(x) (1 + \mathcal{R}(x, y_0))} \leq \e^{- \pi |k| |y_0| \mu(x)}.
\end{equation*}
Factoring out the dominant decay at the boundary $x=R$ gives
    \begin{align*}
    \int_{-R}^R |f(x + \im y_0)| \, \e^{- \pi |k| |y_0| \mu(x)} \e^{2\pi \xi y_0}  \mathrm{d}x&\leq 
       \e^{2\pi |\xi| |y_0|}   \e^{-\pi |y_0| \mu(R) |k|} \int_{-R}^R |f(x + \im y_0)|  \d x.
    \end{align*}
    By our specific choice of the contour shift $|y_0| =\min\{ D,\frac{1}{|\xi|}\}$, the frequency-dependent prefactor is uniformly bounded by $ \e^{2\pi}$. Thus, by $\mu(R)\geq C_\alpha 2^{-\alpha/2}R^{-\alpha}  $ (assuming $R\geq 1$)
    \begin{align*}
    \int_{-R}^R |f(x + \im y_0)| \, \e^{- \pi |k| |y_0| \mu(x)} \e^{2\pi \xi y_0}  \mathrm{d}x &\leq \e^{2\pi }   \e^{-\pi |k||y_0| \mu(R)} \norm{f}_{\F_v} \int_{-R}^R \e^{- r|x|^a}  \d x \\
       &\leq \e^{2\pi }   \e^{-2^{-\alpha/2} \pi |y_0| |k| C_{\alpha}R^{-\alpha}}  \norm{f}_{\F_v} \frac{2}{a \left(\frac{r}{2}\right)^{1/a}} \Gamma\left(\frac{1}{a}\right).
    \end{align*}

    \paragraph{Step 3: Error balancing.} 
    We combine the segments to estimate $|c_k(g)|$. To minimize this upper bound, we match the dominating exponents by setting 
    \begin{align*} -\frac{r}{2} R^a &= -\pi 2^{-\alpha/2} |y_0| C_{\alpha}R^{-\alpha} |k|\\
      R^{a+\alpha} &= |k|\frac{2^{-\alpha/2+1} \pi |y_0| C_{\alpha}}{r} \\
      R &= \left(|k|\frac{2^{-\alpha/2+1} \pi |y_0| C_{\alpha}}{r}\right)^{\frac{1}{a+\alpha}}.
    \end{align*}
    Substituting $R(k)$ back into the common exponential bound yields the fractional spectral rate. Note that $\frac{r}{2} R^a = \frac{r}{2} \left(|k|\frac{2^{-\alpha/2+1} \pi |y_0| C_{\alpha}}{r}\right)^{\frac{a}{a+\alpha}} = \left( \left(\frac{r}{2}\right)^\alpha (2^{-\alpha/2}\pi |y_0| C_\alpha |k|)^a \right)^{\frac{1}{a+\alpha}}$. Thus, the exponential term becomes
    \begin{align*}
        \exp\left(-\frac{r}{2} R^a \right) &= \exp\left(-\left(\left(\frac{r}{2}\right)^{\alpha} \left(2^{-\alpha/2} \pi |y_0| C_{\alpha}|k|\right)^a\right)^{\frac{1}{a+\alpha}}\right).
    \end{align*}
    Summing the contributions from $\Gamma_{\text{tail}}$, $\Gamma_{\text{center}}$, and $\Gamma_{\text{vert}}$, and factoring out the common terms, yields the assertion. Since $\e^{-r R^a} = (\e^{-\frac{r}{2} R^a})^2$, the vertical segment decays twice as fast and is bounded by the common exponential term.
\end{proof}

With this decay of the Fourier coefficients we are able to state a result about the approximation error in the next theorem.

\begin{theorem}
\label{thm:quad_error_exp}
Let $v(x)  = \e^{r|x|^a}$ with $r,a>0$ be a weight function and let the assumptions of Lemma \ref{lem:fourier_decay_exp} hold. 
For $\xi\in \R$ choose $|y_0|= \min\{D,|\xi|^{-1}\}$ and let $M$ be big enough, such that
\begin{equation}\label{eq:bed_M}
M\geq \left(\frac{\ln(2)}{C_{\exp}\,(2^{\frac{a}{\alpha+a}} )}\right)^{\frac{\alpha+a}{a}}  |y_0|^{-1}.
\end{equation}
Then, the error of the approximation~\eqref{eq:ftilde} is bounded by
\begin{equation*}
|\hat{f}(\xi) - \tilde{f}(\xi)| \leq C \norm{f}_{\F_v} \exp\left( - C_{\exp} |y_0|^{\frac{a}{a+\alpha}} M^{\frac{a}{a+\alpha}} \right),
\end{equation*}
where \begin{align*}
    C&=4\left(\frac{2}{a \left(\frac{r}{2}\right)^{1/a}} \Gamma\left(\frac{1}{a}\right) \left(1 + \e^{2\pi } \right) + 2 D \right),\\
    C_{\exp}&= \left(\left(\frac{r}{2}\right)^{\alpha} \left(2^{-\alpha/2}\pi  C_{\alpha}\right)^a\right)^{\frac{1}{a+\alpha}}.
\end{align*}
\end{theorem}
\begin{proof}
We use Theorem~\ref{thm:aliasing_midpoint}, apply the triangle inequality and insert the bound for the Fourier coefficients derived in Lemma \ref{lem:fourier_decay_exp}. Let $|y_0| \coloneqq \min\left\{D, \frac{1}{|\xi|}\right\}$ and define the fractional exponent $\nu \coloneqq \frac{a}{a+\alpha}$. 
    
    The prefactor from Lemma \ref{lem:fourier_decay_exp} can be uniformly bounded for all $\xi \in \mathbb{R}$ by a constant $\widetilde{C} \coloneqq \frac{2}{a \left(\frac{r}{2}\right)^{1/a}} \Gamma\left(\frac{1}{a}\right) \left(1 + \e^{2\pi } \right) + 2 D $, since $|y_0| \leq D$. Substituting $kM$ into the exponential decay term yields
    \begin{align*}
        |I_M(g) - c_0(g)| &\leq \sum_{k \in \mathbb{Z} \setminus \{0\}} |c_{k M}(g)| 
        \leq 2 \widetilde{C} \norm{f}_{\F_v} \sum_{k=1}^\infty \exp\left( - C_{\exp} |y_0|^\nu (k M)^\nu \right).
    \end{align*}
    To isolate the dominant asymptotic behavior, we factor out the leading term ($k=1$) from the infinite sum,
    \begin{equation*}
        \sum_{k=1}^\infty \e^{- C_{\exp} |y_0|^\nu M^\nu k^\nu} = \e^{-  C_{\exp}  |y_0|^\nu M^\nu} \left( 1 + \sum_{k=2}^\infty \e^{- c |y_0|^\nu M^\nu (k^\nu - 1)} \right).
    \end{equation*}
        For the involved sum we calculate, since $k^\nu-1\geq (2^\nu-1)(k-1)$ for $0<\nu<1$, and using the geometric series,
    \begin{align*}
          1 + \sum_{k=2}^\infty \e^{-  C_{\exp}  |y_0|^\nu M^\nu (k^\nu - 1)} &\leq  1 + \sum_{k=2}^\infty \e^{-  C_{\exp}  |y_0|^\nu M^\nu (2^\nu -1)(k - 1)}\\
          &\leq \frac{1}{1-\exp{\left(-  C_{\exp}  |y_0|^\nu M^\nu (2^\nu -1)\right)}}\leq 2,
    \end{align*}
        if \eqref{eq:bed_M} is fulfilled.
    Consequently, the infinite sum is entirely governed by its first mode, preserving the fractional spectral decay rate of the individual Fourier coefficients. This results in the final error bound
    \begin{equation*}
        |I_M(g) - c_0(g)| \leq 4 \widetilde{C}  \norm{f}_{\F_v} \exp\left( - C_{\exp} |y_0 |^{\frac{a}{a+\alpha}} M^{\frac{a}{a+\alpha}} \right),
    \end{equation*}
    which finishes the proof.
\end{proof}
The previous theorem shows that the best error decay rate is achieved if the exponent $\frac{a}{a+\alpha}$ is as large as possible, which is the case for $\alpha\rightarrow 1$. 
The boundary case $\alpha=1$ does not belong to a proper density $\mu$. Especially, in the case of exponential 
decay $a=1$ we receive root-exponential decay $\mathcal{O}(\e^{-C_{\exp}M^{1/2}})$, 
which is up to possible different constants compatible with~\cite{Stenger93,EhGrKl24}. Since $\alpha\rightarrow 1$ also leads to larger constants, in numerical examples we found that a choice which leads to best results is $\alpha\approx 1.3$.
\par

To finally estimate the $L_p$ error of our approximation~\eqref{eq:L_p-error}, we use the splitting of the frequency domain from ~\eqref{eq:Wabfall}.

\begin{corollary}\label{cor:error_decay_exp}
    Let $v(x)  = \e^{r|x|^a}$ and $w(\xi) = \e^{s|\xi|^b}$ with $r,s,a,b>0$ be weight functions and let the assumptions of Lemma \ref{lem:fourier_decay_exp} hold. 
Let $M$ be big enough, such that Theorem~\ref{thm:quad_error_exp} holds. 
    Then the $L_p$-error defined in~\eqref{eq:L_p-error} is bounded by 
\begin{equation*}
E_p^p
\leq  \exp\left( - p\,C_{\exp} M^{\frac{ab}{b(a+\alpha) + a}} \right)\left( 2M^{\frac{a}{b(a+\alpha) + a}}  \,C^p\norm{f}^p_{\F_v} +  \norm{\hat f}_{\F_{w}}^p \frac{2}{b(sp)^{1/b}}\Gamma\left(\frac{1}{b}\right)\right) ,
\end{equation*}
where \begin{align*}
    C&=4\left(\frac{2}{a \left(\frac{r}{2}\right)^{1/a}} \Gamma\left(\frac{1}{a}\right) \left(1 + \e^{2\pi } \right) + 2  \right),\\
    C_{\exp}&= \left(\left(\frac{r}{2}\right)^{\alpha} \left(2^{-\alpha/2}\pi  C_{\alpha}\right)^a\right)^{\frac{1}{a+\alpha}}.
\end{align*} 
\end{corollary}
\begin{proof}
By \eqref{eq:ftilde} and by \eqref{eq:Wabfall} we obtain for the $L_p$-error $E_p$ defined in~\eqref{eq:L_p-error} the estimate 
\begin{align*}\label{eq:L_2error}
E_p^p
&\leq\underbrace{\norm{\hat f}_{\F_{w}}^p \left(\int_{-\infty}^{-\Omega} \frac{1}{|w(\xi)|^p} \d \xi + \int_{\Omega}^\infty \frac{1}{|w(\xi)|^p} \d \xi \right) }_{E_{\text{tail}}}+  \underbrace{\int_{-\Omega}^\Omega|\hat f(\xi) - \tilde f(\xi)|^p \d \xi}_{E_{\text{mid}}} \notag .
\end{align*}
For the tail integrals we have in the case of (sub-)exponential decay,
\begin{align*}
{E_{\text{tail}}} &= \norm{\hat f}_{\F_{w}}^p  \int_{|\xi|\geq \Omega} \e^{-sp|\xi|^{b}} \d \xi=\norm{\hat f}_{\F_{w}}^p \e^{-sp\Omega^{b}} \int_{\R} \e^{-sp|\xi|^{b}} \d \xi \\
&=\norm{\hat f}_{\F_{w}}^p \e^{-sp\Omega^{b}} \frac{2}{b(sp)^{1/b}}\Gamma\left(\frac{1}{p}\right).
\end{align*}
For the middle integral we use Theorem~\ref{thm:quad_error_exp},
\begin{small}
\begin{align*}
    E_{\text{mid}} \leq C^p \norm{f}^p_{\F_v}  \int_{-\Omega}^\Omega   \exp\left( - C_{\exp} p |y_0|^{\frac{a}{a+\alpha}} M^{\frac{a}{a+\alpha}} \right) \d \xi
\end{align*}
\end{small}
To find the optimal truncation parameter $\Omega$ in terms of the number of quadrature nodes $M$, we balance the exponential decay of $E_{\text{tail}}$ and $E_{\text{mid}}$. 

For the middle integral, recall that the contour shift depends on the frequency via $|y_0| = \min\left\{D, \frac{1}{|\xi|}\right\}$. Assuming $M$ is sufficiently large such that the optimal $\Omega$ satisfies $\Omega > 1/\sqrt{D}$, the minimum damping in the interval $[-\Omega, \Omega]$ occurs at the boundaries $\xi = \pm \Omega$. Defining the fractional exponent $\nu \coloneqq \frac{a}{a+\alpha}$, we can bound the integral by its maximum integrand multiplied by the domain size,
\begin{align*}
    E_{\text{mid}} &\leq C^p \norm{f}^p_{\F_v} \int_{-\Omega}^\Omega \exp\left( - C_{\exp} p \min\left\{D, \frac{1}{|\xi|}\right\}^\nu M^\nu \right) \d \xi \notag \\
    &\leq 2\Omega \, C^p \norm{f}^p_{\F_v} \exp\left( - C_{\exp} p \Omega^{-\nu} M^\nu \right).
\end{align*}
To equilibrate the two error contributions $E_{\text{tail}}$ and $E_{\text{mid}}$, we match their dominant exponential decay rates,
\begin{equation*} \label{eq:balance_exponents}
    s p \Omega^b = C_{\exp} p \, \Omega^{-\nu} M^\nu.
\end{equation*}
Solving for $\Omega$ yields the optimal asymptotic scaling for the frequency truncation,
\begin{equation*} \label{eq:omega_opt}
    \Omega(M) = \left( \frac{C_{\exp}}{sp} M^\nu \right)^{\frac{1}{b + \nu}} = \mathcal{O}\left( M^{\frac{a}{b(a+\alpha) + a}} \right).
\end{equation*}
Substituting this optimal choice of $\Omega(M)$ back into the exponential bound of $E_{\text{tail}}$ (or equivalently $E_{\text{mid}}$) yields the final asymptotic rate for the total $L_p$-error,
\begin{equation*}
    E_p^p \leq   \exp\left( - pC_{\exp} M^{\frac{ab}{b(a+\alpha) + a}} \right)\left(2M^{\frac{a}{b(a+\alpha) + a}}C^p\norm{f}^p_{\F_v} +  \norm{\hat f}_{\F_{w}}^p \frac{2}{b(sp)^{1/b}}\Gamma\left(\frac{1}{p}\right)\right) .
\end{equation*}
We bound the chosen constant $D$ from Theorem~\ref{thm:quad_error_exp} by $1$ and receive the assertion.
\end{proof}
As in Section~\ref{sec:polynomial}, the previous theorem shows error decay results for the case where we use variance $\sigma=1$ for the density $\mu$. The additional variance parameter $\sigma $ does not change the error decay rate, but gives the opportunity to decrease the involved preasymptotic constants and therefore also the approximation error.

\paragraph{Comparison with equispaced sampling}
Corollary~\ref{cor:error_decay_exp} states that with our non-equispaced sampling, according to a polynomial distribution $\mu$ with parameter $\alpha$, we receive error decay rates for the error of the approximation of the Fourier transform of $E_p  = \mathcal O\left(\exp\left( - C_{\exp} M^{\frac{ab}{b(a+\alpha) + a}} \right)\right)$. This notation is hiding some polynomial factor in $M$, but the exponential factor decreases much faster. Additionally the exponent $\frac{ab}{b(a+\alpha) + a}$ might be slightly smaller than the exponent in~\cite{EhGrKl24}, but with the parameters $\alpha$ and $\sigma$ we are able to tune the constants to receive smaller approximation errors numerically. While in~\cite{EhGrKl24} the results focus on the discrete $L_2$-error, our estimates are applicable to more general error measures.   

\paragraph{Optimizing the variance $\sigma$}
The best error decay rate for exponential decay in the last subsection does not depend on the variance parameter $\sigma$, however, this parameter can be tuned to reduce the involved constants in the error bounds. \par
As in the polynomial case in Section~\ref{sec:sigma} the parameter $\sigma$ changes the norms $\norm{f}_{\F_v}$ and $\norm{\hat f}_{\F_w}$ in Corollary~\ref{cor:error_decay_exp}. Again, it is not possible to optimize the parameter $\sigma$ directly. Instead we again tune this parameter in numerical application to reduce the pre-asymptotic constant and to receive a small approximation error. 

\section{Numerical simulations}\label{sec:numerics}
In this section, we present numerical experiments to validate the theoretical error bounds and the asymptotic convergence rates derived in the preceding sections. Specifically, we aim to confirm the optimal error decay rate for polynomial function decay.\par

To ensure a rigorous evaluation and to facilitate a direct comparison with recent developments in numerical harmonic analysis~\cite{EhGrKl24}, we adopt the experimental framework and the specific family of test functions introduced by them.
In all subsequent experiments, we compute the following errors to compare the performance,
\begin{align*}
 E_\infty &\coloneqq\sup_{\xi \in [-\Omega,\Omega]}|\hat f(\xi)-\tilde f(\xi)|,\\
 \text{RMSE}  &\coloneqq \left(\frac{1}{|\mathcal X_{\text{test}}|}\sum_{\xi_{i}\in \mathcal X_{\text{test}}} |\hat f(\xi_i)-\tilde f(\xi_i)|^2\right)^{1/2},\\
 \text{MAE}  &\coloneqq \frac{1}{|\mathcal X_{\text{test}}|}\sum_{\xi_{i}\in \mathcal X_{\text{test}}} |\hat f(\xi_i)-\tilde f(\xi_i)|,
 \end{align*}
as a function of the grid size $M$, where $\mathcal X_{\text{test}}\subset[-\Omega,\Omega]$ is a fine test grid. Note that the root-mean-square error (RMSE) and the mean absolute error (MAE) are discretized versions of the $L_2$- and $L_1$-error, respectively.\par

For the implementation of our method we use the polynomial density~\eqref{eq:mu_pol}. Its cumulative distribution function is, up to rescaling, that of a Student-$t$ distribution with $\alpha-1$ degrees of freedom. Hence the inverse $\Psi^{-1}$, which is needed to generate the sampling nodes $x_j$ in~\eqref{eq:f_tilde_approx}, is available through standard routines. For $\alpha=2$ it reduces to the scaled Cauchy quantile $\Psi^{-1}(y) = \sigma\tan\left(\pi\left(y-\tfrac12\right)\right)$.\par

\paragraph{The equispaced benchmark.}
The estimates in~\cite{EhGrKl24} concern the approximation of $\hat f$ at the equispaced frequencies $\tfrac kp$ by the scaled DFT~\eqref{eq:f_hat_approx_equi}. In order to compare both methods at arbitrary frequencies $\xi$, they extend these values by interpolation. The authors  in~\cite{EhGrKl24} use the sinc-function for interpolation, which arises naturally in Fourier approximation. We also choose this interpolation here for comparison to our results.

\paragraph{Summary of the numerical results.}
In summary, the numerical experiments confirm the theoretical picture developed in this paper. For functions with polynomial decay in space and frequency, non-equispaced sampling driven by the optimized density~\eqref{eq:mu_pol} attains the same asymptotic rates as equispaced sampling with considerably smaller pre-asymptotic errors.
For (sub-)exponentially decay we use also a polynomial distribution $\mu$ and receive also (sub-)exponential decay rates, possibly with a slightly worse decay rate. But the numerical results show that in practice we are able to reduce the approximation error due to the additional tuning parameter, the variance $\sigma$. Our transformation approach thus complements the equispaced theory of~\cite{EhGrKl24}.

\subsection{Polynomial decay in time and frequency domain}
We use the family $f^{a,b}$ of functions with exact polynomial decay $a$ in time and $b$ in frequency, which were introduced in~\cite{EhGrKl24}.
We use infinite linear combinations of shifts of the cardinal  B-splines   
defined by  $B_1\coloneqq{\bm 1}_{[-\frac{1}{2},\frac{1}{2})}$, and 
\begin{equation*}
B_{b+1}\coloneqq B_{b}*B_1,\qquad b=1,2,\ldots\,.
\end{equation*}
 Each $B_b$ is a piecewise polynomial function of degree $b-1$ that is
 $b-2$-times continuously differentiable with support 
 $[-\frac{b}{2},\frac{b}{2}]$, and its  Fourier transform is 
 $\widehat{B_b}(\xi ) = \big( \tfrac{\sin \pi \xi }{\pi \xi } \big)^b
 = \sinc ^b (\xi )$, see~\cite[Section 9.1]{PlPoStTa23}.  

Consider $(c_k)_{k\in\Z}\subseteq \C$ given by its nonzero entries
\begin{equation*}
c_0=\tfrac{1}{2},\qquad c_k = \tfrac{1}{k\pi \mathrm{i}}, \text{ for }
k\in 2\Z+1\, , 
\end{equation*}
so that $c_{2k}=0$.
For integer parameters $a,b \in \N$  we study the family of  functions
\begin{equation}\label{eq:fab}
f^{a,b}(x) \coloneqq \sum_{k\in \Z} c^a_k \,B_b(x-k),\qquad a,b=1,2,\ldots \, .
\end{equation}
Every $f^{a,b}$ is a  locally finite sum,  therefore its  point
evaluations can be computed accurately in numerical experiments, and they satisfy $\sup_{x\in\R}|f^{a,b}(x)| (1+x^2)^{a/2}<\infty$, such that $f^{a,b}\in \F_{v_a,0}$ with $v_a=(1+x^2)^{a/2}$. Furthermore, since each $B_b$ is
 $b-2$-times continuously differentiable, we also have $f^{a,b}\in \F_{v_a,r}$ for all $r\in \N$ with $r\leq b-2$.

To compute the Fourier transform of $f^{a,b}$, we define the Fourier
series $u_a(\xi)\coloneqq \sum_{k\in \Z} c^a_k\e^{-2\pi\im k\xi}$ and
obtain 
\begin{align*}
\widehat{f^{a,b}}(\xi) 
& = \sinc(\xi)^b \,  \sum_{k\in \Z} c^a_k\e^{-2\pi\im k\xi}=\sinc(\xi)^b u_a(\xi)\,.
\end{align*}
The estimate $|\sinc^b(\xi)| \lesssim (1+\xi^2)^{-b/2}$ implies
$\sup_{\xi\in\R}|\widehat{f^{a,b}}(\xi)|(1+\xi^2)^{b/2}<\infty$, thus $f^{a,b}\in \F_{w}$ with $w=(1+\xi^2)^{b/2}$.

Following~\cite{EhGrKl24}, the functions $u_a$ are periodic functions with
period $1$ whose values on $[-1/2,1/2)$ for $a=1,\ldots,5$ are given by
\begin{align*}
u_1(\xi) &={\bf
  1}_{(-\frac{1}{2},0)}(\xi)+\frac{1}{2}{\bf
  1}_{\{-\frac{1}{2},0\}}(\xi)\\
u_2(\xi)  &= |\xi|,\\
u_3(\xi) &= -\sign(\xi) \xi^2+\frac{1}{2}\xi+\frac{1}{8},\\
u_4(\xi) &= 2|\xi|^3/3 - \xi^2/2  + 1/12,\\
u_5(\xi) &= -\sign(\xi)\xi^4/3 + \xi^3/3 - \xi/24+1/32  . 
\end{align*}
Thus, we can compute the point evaluations of $f^{a,b}$ and
$\widehat{f^{a,b}}$ exactly, both analytically and numerically, and
the decay parameters are precisely $a$ and $b$. 
For the numerical simulations we compute the samples
$\widehat{f^{a,b}}(\xi_j)$ and compare them with our approximation $\tilde f(\xi_j)$.\par

In Figure~\ref{fig:poly_numerics} we plot the approximation error for our transformation method and compare against the results from~\cite{EhGrKl24}. Note that instead of the guarantees for the discrete error at equispaced frequencies as in~\cite{EhGrKl24}, our results cover the approximation error in the frequency domain in $L_p$. For that reason we show the numerical approximation error on a fine test grid of size $|\mathcal X_{\text{test}}|=10^4$. Both methods exhibit the same asymptotic decay rates, in accordance with the analysis in Section~\ref{sec:polynomial}, but over the whole range of $M$ our non-equispaced method achieves noticeably smaller errors than the equispaced benchmark, confirming the improved pre-asymptotic constants.

\begin{figure}[htb]
    \centering
    \includegraphics[width=1\linewidth]{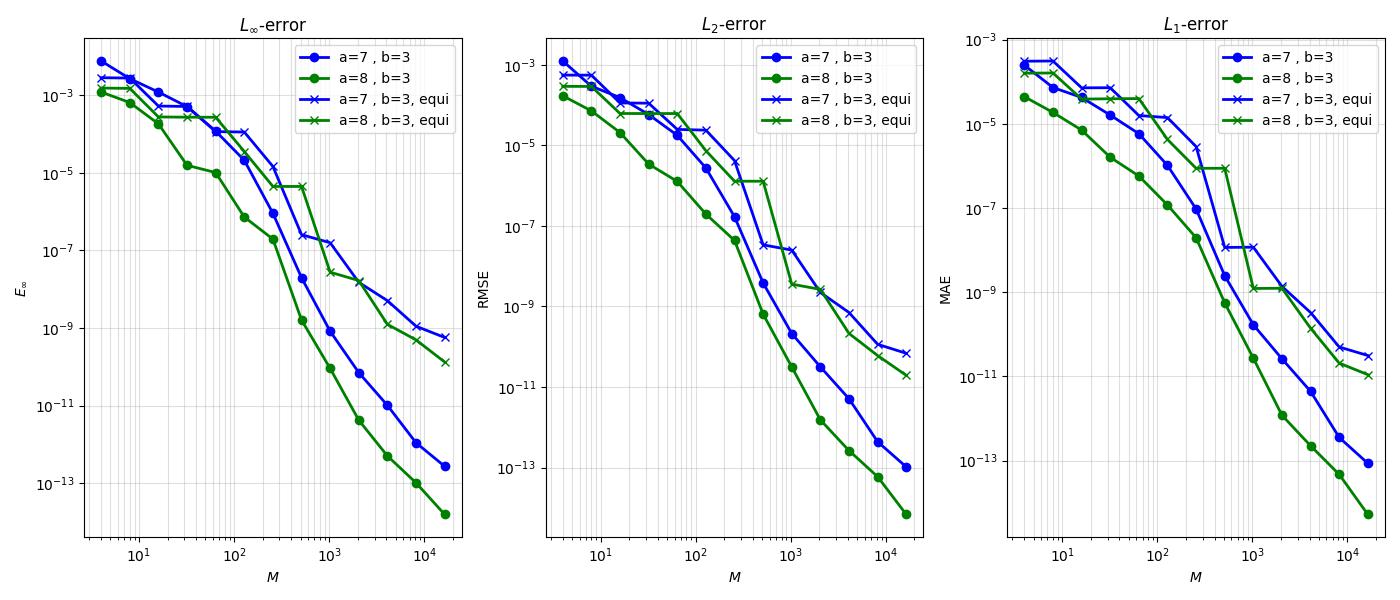}
    \caption{Approximation of the Fourier transform of the function~\eqref{eq:fab}: Comparison of our method with equispaced sampling from~\cite{EhGrKl24}. }
    \label{fig:poly_numerics}
\end{figure}

\subsection{Optimality of the density parameters $\alpha$ and $\sigma$}\label{sec:num_params}
We now confirm numerically the two parameter choices derived in Section~\ref{sec:polynomial}. The optimal tail exponent $\alpha$ from~\eqref{eq:alpha_opt} discussed in Section~\ref{sec:alpha} and the variance $\sigma$ discussed in Section~\ref{sec:sigma}. In both experiments we use the test function $f^{a,b}$ from~\eqref{eq:fab} with $a=7$, $b=3$, so that we can choose $m=2$ in Theorem~\ref{thm:error_hatf}. The corresponding results are shown in Figure~\ref{fig:params}.\par

For these parameters the optimal exponent~\eqref{eq:alpha_opt} equals $\alpha^\ast = \tfrac{2(b-m)(a-1/2)}{m(2b-1)} = 1.3$. The left plot of Figure~\ref{fig:params} shows the RMSE as a function of $M$ for several values of $\alpha$. In agreement with the analysis, the error decay rate is the largest if the optimal parameter $\alpha=\alpha^\ast$ is chosen.\par

The right plot illustrates the effect of the variance $\sigma$ at the fixed exponent $\alpha=\alpha^\ast$. As predicted, $\sigma$ does not change the asymptotic rate but acts as a multiplicative prefactor on the error constant. Tuning $\sigma$ reduces the RMSE, the optimal variance may depend on $M$, but seems to be not far from $\sigma=1$. Although the optimal $\sigma$ is function-dependent and not available in closed form, a simple one-dimensional search over $\sigma$ yields a substantial reduction of the pre-asymptotic constants.

\begin{figure}[htb]
    \centering
    \includegraphics[width=1\linewidth]{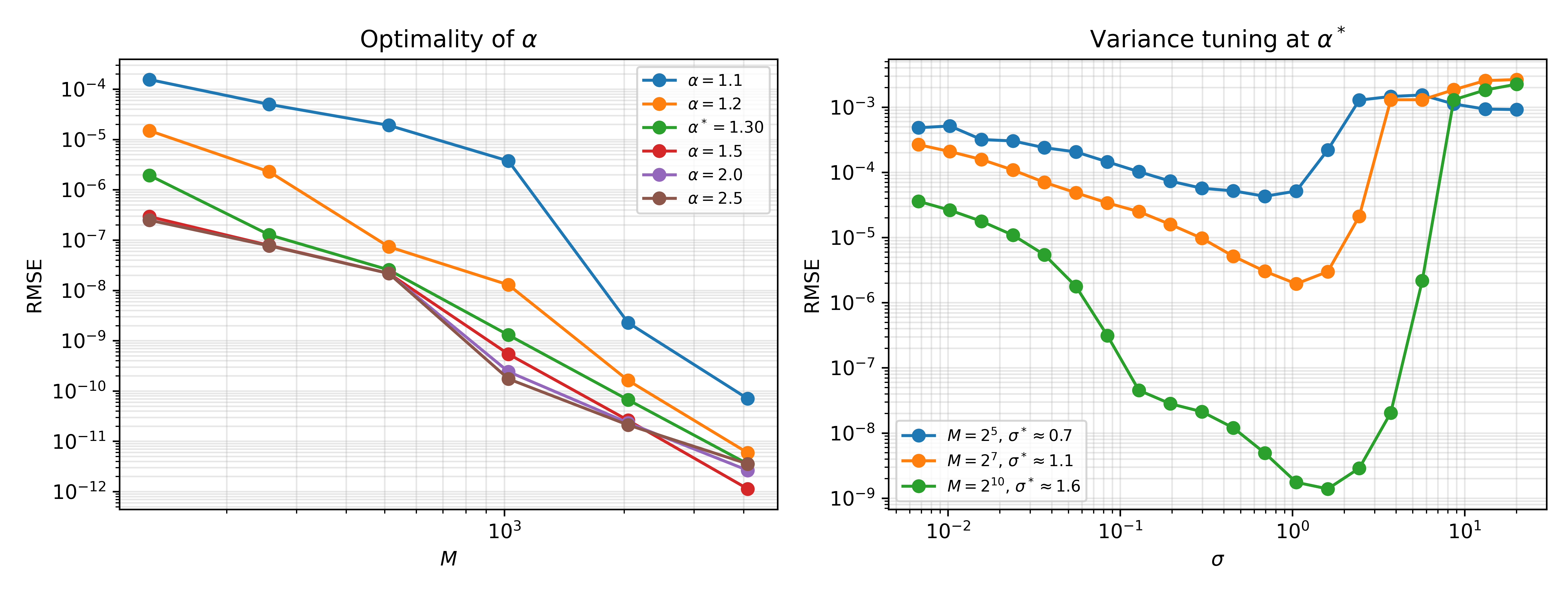}
    \caption{Confirmation of the parameter choices of Section~\ref{sec:polynomial} for the test function~\eqref{eq:fab} with $a=7,b=3$. Left: RMSE versus $M$ for several tail exponents $\alpha$. Right: RMSE versus the variance $\sigma$ at $\alpha=\alpha^\ast$ for several values of $M$.}
    \label{fig:params}
\end{figure}

\subsection{Exponential decay in time and polynomial in frequency}\label{sec:num_exp}
To numerically evaluate the robustness of the proposed density mapping approach against classical equispaced grids, we consider the following test function,
\begin{equation}\label{eq:f_sing}
    f_{\varepsilon}(x) = \mathrm{e}^{-\sqrt{x^2 + \varepsilon^2}}, \quad \text{with } \varepsilon = 10^{-2}.
\end{equation}
Its exact Fourier transform is given analytically in terms of the modified Bessel function of the second kind, $K_1$,
\begin{equation*}
    \hat{f_\varepsilon}(\xi) = \frac{2\varepsilon}{\sqrt{1+4\pi^2\xi^2}} K_1\big(\varepsilon\sqrt{1+4\pi^2\xi^2}\big).
\end{equation*}

Asymptotically, the function exhibits purely exponential decay in both the spatial and frequency domains. In terms of exponential decay we calculate the parameters
\begin{equation*}
    a = 1, \quad r = 1, \quad b = 1, \quad s = 2\pi\varepsilon.
\end{equation*}
The numerical results are shown in Figure~\ref{fig:exp}. We choose the polynomial density $\mu$ from~\eqref{eq:mu_pol} with parameter $\alpha=1.3$. For different numbers of samples $M$ we plot the RMSE, MAE and maximal error over $1000$ equispaced test samples in $[-100,100]$. We tune the variance parameter $\sigma$ for every $M$ separately. For comparison with~\cite{EhGrKl24} we use the parameters proposed there for the equispaced approximation. Our approach achieves lower errors.\par 

Geometrically, the small parameter $\varepsilon$ induces a sharp peak (a near-singularity) at the origin, while preserving the heavy exponential tails $\mathcal{O}(\mathrm{e}^{-|x|})$ for $|x| \to \infty$. This structural property exposes a weakness of the optimal equispaced grid. The sharp gradient near the origin forces the optimal theoretical step size $h$ to be small. Consequently, the truncated spatial interval remains extremely small for a moderate number of quadrature points $M$. \\

In contrast, our proposed non-equispaced approach using a polynomial density mapping overcomes this bottleneck. The density function naturally clusters quadrature nodes at the origin to resolve the near-singularity, while simultaneously stretching the grid sufficiently wide to capture the exponential tails. 
\begin{figure}[htb]
    \centering
    \includegraphics[width=1\linewidth]{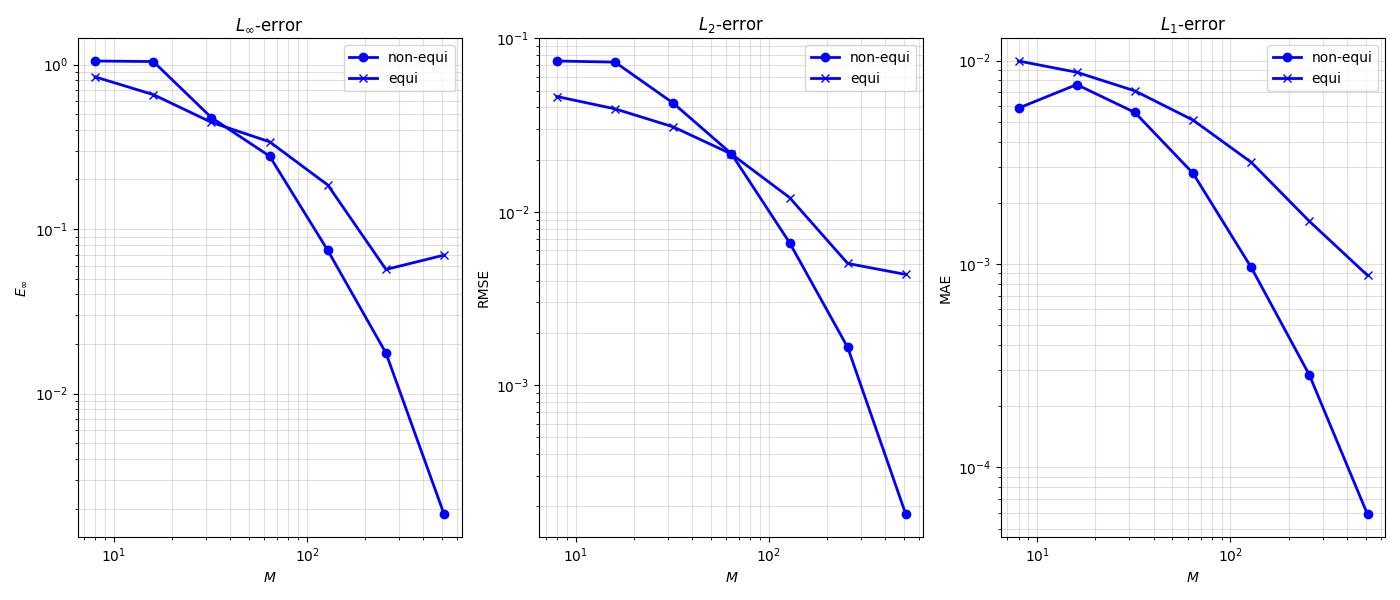}
    \caption{Approximation of the Fourier transform of the function~\eqref{eq:f_sing}: Comparison of our method (non-equi) with equispaced sampling (equi) from~\cite{EhGrKl24}. }
    \label{fig:exp}
\end{figure}

\paragraph*{Funding.} L.W. acknowledges financial support by the Deutsche Forschungsgemeinschaft (DFG, German Research Foundation) - project number 569580074

\appendix
\section{Appendix}
\subsection{Detailed calculations of frequently used integrals}
The following classical integral is used repeatedly throughout the paper. The following result is used for the polynomial case.
\begin{lemma}\label{lem:beta_integral}
For $s>\tfrac 12$ it holds
\begin{equation*}
\int_{-\infty}^\infty (1+x^2)^{-s}\d x = \sqrt{\pi}\, \frac{\Gamma\left(s-\tfrac12\right)}{\Gamma(s)}.
\end{equation*}
\end{lemma}
\begin{proof}
The substitution $x = \tan\theta$ with $\d x = (1+\tan^2\theta)\d\theta$ gives
\begin{equation*}
\int_{-\infty}^\infty (1+x^2)^{-s}\d x = \int_{-\pi/2}^{\pi/2} (\cos\theta)^{2s-2}\d\theta = B\left(s-\tfrac 12,\tfrac 12\right) = \frac{\Gamma\left(s-\tfrac12\right)\Gamma\left(\tfrac12\right)}{\Gamma(s)},
\end{equation*}
where we used the trigonometric representation of the beta function \[B(x,y) = 2\int_0^{\pi/2}(\sin\theta)^{2x-1}(\cos\theta)^{2y-1}\d\theta,\]
see e.g.~\cite[Chapter 1]{andrews1999}. The assertion follows from $\Gamma(\tfrac 12)=\sqrt{\pi}$.
\end{proof}
Lemma~\ref{lem:beta_integral} with $s=\tfrac\alpha2$ and the substitution $x=\sigma y$ yields
\begin{equation*}
\int_{-\infty}^\infty \left(1+\frac{x^2}{\sigma^2}\right)^{-\alpha/2}\d x = \sigma\sqrt{\pi}\,\frac{\Gamma\left(\frac{\alpha-1}{2}\right)}{\Gamma\left(\frac{\alpha}{2}\right)},
\end{equation*}
which proves the formula for the normalization constant $C_{\alpha,\sigma}$ in~\eqref{eq:mu_pol} and shows that $\alpha>1$ is necessary for the integrability of the density. Moreover, applying Lemma~\ref{lem:beta_integral} with $s = \tfrac{a+m-\alpha m}{2}$ gives the constant $A_{\mu,a,m}$ in Section~\ref{sec:rates}. The requirement $s>\tfrac 12$ is exactly the admissibility condition~\eqref{eq:cond_alpha}, $\alpha<\tfrac{a-1}{m}+1$.\par
For the theory of the (sub) exponential decay we frequently use the following integral.
\begin{lemma}
For $a, r > 0$ we have
\begin{equation*}
    \int_{-\infty}^\infty \e^{- r|x|^a} \d x = \frac{2}{a} r^{-\frac{1}{a}} \Gamma\left(\frac{1}{a}\right),
\end{equation*}
where $\Gamma(s) = \int_0^\infty t^{s-1} \e^{-t} \d t$ denotes the standard Gamma function.
\end{lemma}
\begin{proof}
Due to the symmetry of the integrand, we can write the integral over $\mathbb{R}$ as twice the integral over the positive half-axis:
\begin{equation*}
    \int_{-\infty}^\infty \e^{-r|x|^a} \d x = 2 \int_0^\infty \e^{-rx^a} \d x.
\end{equation*}
Using the substitution $u = rx^a$ (hence $x = (u/r)^{1/a}$ and $\d x = \frac{1}{a} r^{-1/a} u^{\frac{1}{a} - 1} \d u$) yields
\begin{equation*}
    2 \int_0^\infty \e^{-u} \frac{1}{a} r^{-\frac{1}{a}} u^{\frac{1}{a} - 1} \d u = \frac{2}{a} r^{-\frac{1}{a}} \int_0^\infty u^{\frac{1}{a} - 1} \e^{-u} \d u.
\end{equation*}
Recognizing the integral representation of the Gamma function $\Gamma\left(\frac{1}{a}\right) = \int_0^\infty u^{\frac{1}{a}-1} \e^{-u} \d u$, we obtain the desired result.
\end{proof}

\bibliographystyle{abbrvurl}
\bibliography{references}

\end{document}